\documentclass{IEEEtran}
\IEEEoverridecommandlockouts 
\usepackage{times} 
\usepackage{pdfsync} 
\usepackage{epsfig}
\usepackage{epstopdf}
\usepackage{amsmath}  \interdisplaylinepenalty=2500
\usepackage{amssymb}
\usepackage{color} 
\usepackage{cite} 
\usepackage{array}
\usepackage{mathrsfs}
\usepackage[english]{babel} 
\usepackage{fix2col}

\newcommand{\Rset}{\mathbb{R}}
\newcommand{\Nset}{\mathbb{N}}

\newcommand{\Kset}{\mathbb{K}}

\newcommand{\Bset}{\mathbb{B}}
\newcommand{\Jset}{\mathbb{J}}

\newcommand{\Prob}{\mathbf{Pr}}

\newcommand{\E}{\mathbf{E}}

\newcommand{\Gauss}{\mathcal{N}}

\newcommand{\eq}{\triangleq}
\newcommand{\hfs}{\hfill\ensuremath{\square}}

\DeclareMathOperator{\eigs}{eigs}
\DeclareMathOperator{\diag}{diag}
\DeclareMathOperator{\tr}{tr}
\DeclareMathOperator{\node}{Node}
\DeclareMathOperator{\edge}{Edge}

\DeclareMathOperator{\parent}{Par}
\DeclareMathOperator{\path}{Path}

\newcommand{\Ecal}{\mathcal{E}}

\newcommand{\Ical}{\mathcal{I}}

\newtheorem{thm}{Theorem} 
\newtheorem{defi}{Definition}
\newtheorem{coro}{Corollary} 
\newtheorem{lem}{Lemma} 
\newtheorem{rem}{Remark}
\newtheorem{ex}{Example} 
\newtheorem{ass}{Assumption} 

\title{State Estimation over Sensor
  Networks with Correlated Wireless Fading Channels}

\author{Daniel~E.~Quevedo,~\IEEEmembership{Member,~IEEE,}
  Anders~Ahl\'en,~\IEEEmembership{Senior Member,~IEEE,}\\
  and Karl H.\ Johansson,~\IEEEmembership{Senior Member,~IEEE}%
\thanks{%
    Daniel Quevedo is
  with the School of Electrical Engineering \&
  Computer Science, The University of Newcastle, NSW
  2308, Australia; dquevedo@ieee.org. 
  Anders~Ahl\'en is with Signals and Systems, Uppsala University,  SE-751 21,
  Uppsala, Sweden; Anders.Ahlen@signal.uu.se. 
  Karl H.\ Johansson is with ACCESS Linnaeus Centre, School of Electrical
  Engineering, Royal Institute of 
  Technology, Stockholm, Sweden; kallej@ee.kth.se}
\thanks{This research was supported in part under Australian Research Council's
  Discovery Projects funding scheme (project number DP0988601), and
  by the VINNOVA project WiComPI, project Dnr2009-02963. A preliminary version
  of parts of this work was presented as\cite{queahl11b}.}}

\begin{document}

\maketitle

\begin{abstract}               
Stochastic stability for centralized time-varying  Kalman filtering over a wireless
sensor network with correlated fading channels is studied. On their
route to the gateway, sensor packets, possibly 
aggregated with measurements from several nodes, may be dropped because of
fading links. To study this situation, we introduce a network state  process, which describes a finite
set of configurations 
 of the radio environment. The network state characterizes the channel gain
 distributions  of the
  links, which are allowed to be correlated between each
 other. Temporal correlations of channel gains are modeled by allowing the network state
 process to form a (semi-)Markov chain. We establish
sufficient conditions that ensure the Kalman filter to be exponentially bounded. In the one-sensor case, this new stability condition is shown to 
include previous results obtained in the literature as special cases. The
results also hold when using power and bit-rate control policies, where the transmission
power and bit-rate of each node  are nonlinear mapping of the network state and 
channel gains.  
\end{abstract}


\section{Introduction}
 \label{sec:introduction}
Wireless sensor technology is of growing interest for process and
automation industry. The driving force behind using wireless
technology in monitoring and control applications is its lower
deployment and reconfiguration cost.
In addition, wireless devices can be placed where wires 
cannot go, or where power sockets are not available; see, e.g.,
\cite{chejoh11,aluinn11,bjonet11,ilymah04,shezha07}.   
\par A drawback of  wireless communication
technology lies in that wireless channels are subject to
fading and interference, which frequently lead to packet errors. 
The wireless channel is in general time varying. This time-variability
may in an industrial setting be caused by 
moving machines, vehicles, people, and so forth, or when the receiver
or the transmitter are mounted on a moving object. Therefore, in
addition to the propagation path loss, channels will commonly be
subject to shadow and small-scale fading
\cite{goldsm05,proaki95}. The channel fading 
 can be partially compensated for through control of bit-rates and the power
 levels used by the radio amplifiers; see,
 e.g.,\cite{caitar99,bergal02,queahl10,queahl12}. The loss of information due to
 channel fading is one of the main problems of wireless estimation and control
 systems, leading to stability and performance degradation.

\par Several interesting approaches have been reported for 
state estimation of  linear time-invariant (LTI) systems via  wireless sensor
networks. For example, the works \cite{shicap10} and \cite{shi09b}
focus on delay issues in a 
sensor network with no dropouts, whereas \cite{gupdan09} studies the
effect of dropouts within an architecture  with only one sensor node, but
where additional relay nodes are allowed to process data. The paper
\cite{chisch11} examines various information fusion strategies for distributed
state estimation in sensor
networks having a star topology.  In the recent 
work\cite{mogar11}, the authors examine sensor scheduling for 
networks with a tree topology and no dropouts. 
 The issue of Kalman filter stability (i.e.,
boundedness of the estimation error covariance matrix) has received significant
attention. In particular, \cite{sinsch04} focused on LTI plants and a
single-link architecture 
where dropout
processes are independent and identically distributed (i.i.d.). By using a
fixed-point argument, \cite{sinsch04}  established
that there exists a critical dropout probability value which separates
situations where the expected value of the estimator covariance matrix remains bounded
from instances where it diverges, see
also \cite{plabul09,karsin12} and \cite{liugol04}.\footnote{If, instead of the expected value of the
  covariance matrix, other criteria are used, then different critical dropout
  probabilities will be obtained\cite{censi11,huadey07,rohmar10}. Alternatively, the
  works\cite{censi11,epsshi08,shieps10,karsin12} directly examine the distribution of the
  covariance matrix.} The latter work  examines  a
state estimation architecture with two 
channels affected 
by  i.i.d.\ dropouts. The case  where
 dropouts are described 
by  a time-homogeneous two-state Markov chain
was investigated in \cite{huadey07,jingup06,xiexie08,shieps10,youfu11}.
Recently,\cite{censi11} studied the one-sensor case with dropouts governed
by a, more general, 
semi-Markov chain. \emph{Inter-alia}, that work combined results of
\cite{stenfl96}  with bounds established in 
\cite{bouger93} to show that,  under mild conditions, the empirical
covariance of the estimation error converges to a unique stationary
distribution.

\par In the present work, we study centralized state estimation for linear time-varying
systems via wireless sensor networks with a tree topology. Communication links
are subject to random and possibly correlated packet dropouts.  Based on
motivating case 
studies from process 
industry, we assume that 
in-network processing is much faster than the dynamics of the system
whose state is being estimated and, thus, neglect delays introduced by
the network.   A key contribution is the proposal of a sensor
network fading model, which allows for spatial and temporal correlations of
channel gains and, thereby, packet dropouts. For that purpose, we introduce
 a \emph{network state} process. The latter models shadow fading effects by describing a finite set of configurations
 of the radio environment. The network state characterizes the channel gain
 distributions  of the
 different links, which depend upon  shadow and multi-path fading. To
 model temporal correlations of channel gains, we allow the network state
 process to form a (semi-)Markov chain. Our radio model structure  generalizes  models
previously reported in the literature; see,
e.g.,\cite{gilber60,xiexie08,bergal02,huadey07,smisei03} and  also accounts for power and
bit-rate control of
sensor radio amplifiers.

\begin{figure}[t]
  \centering
  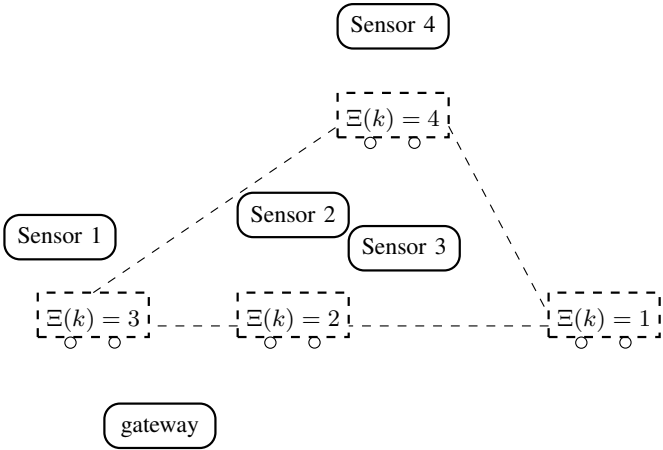
  \caption{Sensor network with moving robot.}
  \label{fig:truck_tac12}
\end{figure}

\par As a motivating example, 
 Fig.~\ref{fig:truck_tac12} schematizes an industrial situation,
 where a mobile robot is moving between four different operational points. The sensor
 network is set up for state estimation at the gateway. This situation can be
 modeled by assigning each robot position to a 
 network state value $\Xi(k)\in\{1,2,3,4\}$. Clearly, different  positions
 will lead to different fading environments, i.e., fading distributions between
 sensors and gateway.

\par  By using elements of stochastic stability theory, we derive sufficient
conditions on the system matrix and network parameters for the trace of the estimator
covariance 
matrix  to be exponentially bounded. For cases where the network state process
forms a Markov Chain, the result obtained  depends upon the transition
probabilities of the network state 
and the conditional probabilities of the associated instantaneous system
observation matrix to 
have full  column rank. For network state processes described by a semi-Markov
Chain, our result depends upon the transition probabilities of the
embedded Markov Chain, conditional holding time distributions, and the
conditional probability of a suitably defined multi-step observability matrix to
have full rank. 
 In special cases, the results obtained correspond to conditions
which have been previously documented in 
the literature on state estimation with packet dropouts. The present paper
expands upon  our recent conference contribution\cite{queahl11b} by   considering
a semi-Markov model for the network state.   

\par The remainder of the paper is organized as follows:
Section~\ref{sec:state-estim-over} describes the sensor network
architecture. The proposed network fading model is presented in
Section~\ref{sec:transmission-effects}. Section~\ref{sec:state-estim-over-1}
characterizes 
the associated state estimator. In
Section~\ref{sec:stability-analysis}, we establish sufficient conditions for
exponential boundedness when the network state process is  Markovian.  Section~\ref{sec:one-sensor-case} studies the special case where the network has only
one sensor. Stability results for state estimation when the network states are
described by a semi-Markov 
chain are derived in
Section~\ref{sec:alternative-model}. Section~\ref{sec:conclusions-1} draws
conclusions. Technical proofs are included in appendices.
 
\subsubsection*{Notation}
\label{sec:notation}
We write $\Nset$ for $\{ 1, 2, \ldots\}$, and $\Nset_0$ for $\Nset\cup\{0\}$;
$\Rset$ are the real numbers, $\Rset_{\geq 0}\eq [0,\infty)$. The notation
$\{\nu\}_{\Nset_0}$ refers to the sequence $\{\nu(0), \nu(1),\dots \}$, and
$\{\nu\}_\ell^k$
to $\{\nu(\ell), \nu(\ell+1), \dots, \nu(k)\}$, with $\{\nu\}_\ell^k=\emptyset$, the empty
set, whenever $\ell>k$. The notation $|\cdot |$ refers to
cardinality of a set. 
Given any vector $v$, its Euclidean norm is denoted $\|v\|=\sqrt{v^Tv}$,
where the superscript $T$ refers to transposition. 
  The trace of a matrix $A$ is denoted by $\tr A$, and its spectral norm  by
$||A||\eq \sqrt{\max \eigs (A^TA)}$, where $\eigs (A^TA)$ are the eigenvalues
of $A^TA$.  If a matrix $A$ is positive definite (non-negative definite), then
we write $A\succ 0$ ($A\succeq 0$); $I_n$ denotes the $n\times n$ identity matrix.
 To denote the conditional probability of an event
$\Omega$ given $\Delta$, we write 
$\Prob\{\Omega \,|\, \Delta\}$. 
The expected value of a random variable $\mu$ given
 $\Delta$, is denoted via  $\E\{\mu  \,|\, \Delta \}$, whereas for the unconditional
 expectation  we write $\E\{\mu\}$. A real Gaussian random variable $\mu$, with
 mean $\nu$ and  covariance
 matrix $\Gamma$ is denoted by $\mu \sim\Gauss(\nu, \Gamma)$. 

\section{Sensor Network Architecture}
\label{sec:state-estim-over}
We consider uncontrolled  linear time-varying $n$-dimensional systems of the form:
\begin{equation}
  \label{eq:1}
  x(k + 1)= A(k)x(k) + w(k), \quad k\in\Nset_0,
\end{equation}
where 
$x(0)\sim\Gauss(x_0,P_0)$, with $x_0^Tx_0\in\Rset_{\geq 0}$, $\|P_0\| \in\Rset_{\geq 0}$. The driving noise process
  $\{w\}_{\Nset_0}$ is independent  with $w(k)\sim\Gauss(0,Q(k))$, for all $k \in \Nset_0$.

\par To  estimate the system state sequence $\{x\}_{\Nset_0}$, a collection
of $M$ wireless sensors $\{S_1,\dots,S_M\}$ is used. Each sensor
provides  
a  noisy measurement sequence $\{y_m\}_{\Nset_0}$ of the form
\begin{equation}
  \label{eq:2a}
  {y}_m(k) = C_mx(k) + v_m(k), \quad m\in \{1,\dots,M\},
\end{equation}
with $C_m\in\Rset^{l_m\times n}$, $l_m\in\Nset$. In~\eqref{eq:2a}, the
measurement noise processes
 $\{v_m\}_{\Nset_0}$ are independent, with each 
$v_m(k)\sim\Gauss (0,R_m(k))$.\footnote{In addition to measurement noise, $v_m(k)$ may also describe quantization
effects, which we model as Gaussian and introducing possibly time-varying
distortion, see also\cite{queahl10,queost11,minfra09}.} Throughout this work, we assume
that $\{A\}_{\Nset_0}$, $\{Q\}_{\Nset_0}$ and 
$\{R_m\}_{\Nset_0}$, $m\in\{1,\dots,M\}$ are deterministic and bounded sequences, known at the gateway.

\par The $M$ (possibly vector) measurements in~(\ref{eq:2a}) are to be transmitted  via  wireless
links to a single 
gateway (or fusion centre), denoted $S_0$. Since the links are wireless, and
independent on whether the medium access protocol adopted gives deterministic or
random access (or a combination thereof  as in the widely used IEEE 802.15.4
communication standard, see\cite[Ch.4]{park11}), some
measurements will be dropped by the network.
The received measurement values are used
to remotely estimate the state of the 
system~(\ref{eq:1}).  The present work seeks to gain
understanding on the impact of measurement dropouts on estimation performance.
For that purpose, we will focus on a sensor network architecture where
all nodes in 
the network, apart 
from 
the gateway, are sensing nodes.\footnote{It is worth
noting, however, that relay nodes can be modeled as a sensor node with an
all-zero 
observation matrix, see~\eqref{eq:2a}.}

\par As in other works, e.g.,\cite{mogar11}, we will assume that the network is much faster than the
process~\eqref{eq:1} and will therefore neglect any delays experienced by the
data when traveling through the network. Each sensor node aggregates its own
current measurement to
the received packets from incoming nodes and transmits the resulting packet to a single destination
node. This reduces the energy used for listening. Sensor nodes do not buffer old 
data. Thus, the measurements received by the gateway
at time $k$ are a subset of $\{y_1(k),y_2(k),\dots,y_M(k)\}$.

\par It is convenient to describe the network by means of a directed graph, with vertices
$\{S_0,\dots , S_M\}$, and edges associated with the  wireless links. Each sensor $S_m$
transmits to a single node, called its \emph{parent} and henceforth denoted via 
$\parent (S_m)$. Thus, the graph constitutes a
directed tree graph with root $S_0$, see also\cite{mogar11}. Each
sensor node
$S_m$ has a single outgoing edge, say, 
\begin{equation*}
\Ecal_m=\big(S_m,\parent (S_m)\big),\quad  m\in\{1,\dots,M\}.
\end{equation*}
 Furthermore, there exists a
unique path from 
each $S_m$ to the gateway. We  denote this path  by $\path(S_m)$,  its
edges by 
$\edge (\path(S_m))$ and its nodes by $\node (\path(S_m))$.


\begin{figure}[t]
  \centering
  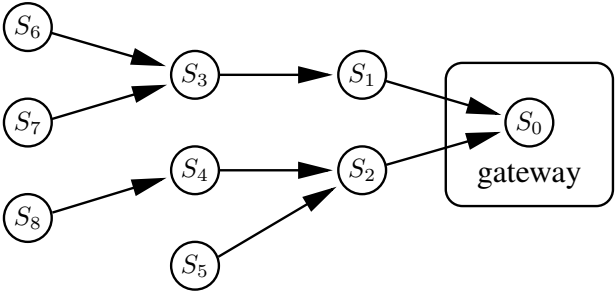
  \caption{Sensor network tree with $9$ nodes and $8$ edges.}
  \label{fig:tree_tac12}
\end{figure}

\begin{ex}
  In the sensor network  depicted in
  Fig.~\ref{fig:tree_tac12}, the packet transmitted
by  $S_3$  at time $k\in\Nset_0$ contains $y_3(k)$ and a subset of
$\{y_6(k),y_7(k)\}$. Furthermore, we have 
$ \parent  (S_4)=\parent (S_5)=S_2$, $\node (\path(S_7)) =
\{S_7,S_3,S_1,S_0\}$, and 
$\edge(\path(S_7))=\{\Ecal_7,\Ecal_3,\Ecal_1\}$ $=\{(S_7,S_3)$, $(S_3,S_1),(S_1,S_0)\}$.  \hfs
\end{ex}

Since the  links used to convey measurements from the 
 sensors to the gateway are wireless, transmission
errors are likely to occur. To study this aspect,  in the
following section we introduce a  network model, which is physically motivated
and leads to tractable analysis of the overall system

\section{Sensor Network Connectivity Model}
\label{sec:transmission-effects}
Channel gains are, in general,  affected by  path-loss, and
shadow and small-scale fading; see, e.g.,
\cite{proaki95,goldsm05,ghamos11}. Path-loss is solely distance dependent and will 
therefore be  constant in most industrial applications. Shadow
fading is 
caused by large (and possibly slowly moving) objects obstructing the
radio link  and can therefore be correlated in time and space; see also\cite{gudmun91,agrpat09}. If sensors are
close  
to each other, then shadow fading may cause correlations between the
individual link gains.  Small-scale fading is due to local scattering  in multi-paths and
 is commonly  modeled via uncorrelated channel gain
distributions.  

\subsection{Network Fading Model}
\label{sec:network-fading-model}
To model the  sensor network fading channels, we will make use of $M+1$ random
variables: The
\emph{network state} process $\{\Xi\}_{\Nset_0}$, and  $M$  channel (power) gains
$\{h_m\}_{\Nset_0}$.  

\par  The  network state process  serves to capture shadow fading. It is a discrete process, i.e., we have 
\begin{equation}
  \label{eq:15}
  \Xi(k)\in \Bset
  \eq 
  \{1,2,\dots ,|\Bset|\},\quad \forall k\in\Nset_0,
\end{equation}
where the finite state space $\Bset$  models
different configurations of the overall physical environment (such as positions
of mobile objects). 
 To incorporate temporal
correlations, throughout the first part of this work, we will assume that
$\{\Xi\}_{\Nset_0}$ is a Markov chain, as stated in Assumption~\ref{ass:Markov} below. In Section~\ref{sec:alternative-model}, we will
extend this model by incorporating arbitrary holding times on system states.
\begin{ass}
\label{ass:Markov}
 The network states $\{\Xi\}_{\Nset_0}$  form a discrete (time-homogeneous)
  Markov chain, with transition probabilities
  \begin{equation}
    \label{eq:5}
    p_{ij} = \Prob\big\{\Xi(k+1) = j \, \big| \, \Xi(k) = i\big\},\quad \forall i,j \in
    \Bset, \;k\in\Nset_0,
  \end{equation}
  see, e.g., \cite{bremau99,cinlar75}. \hfs
\end{ass}

\par Each of the channel gains, $\{h_m\}_{\Nset_0}$, $m\in\{1,2,\dots,M\}$,
corresponds to the power gain from node
$S_m$ to its parent node, $\parent(S_m)$. We consider 
block-fading, which is a 
common information theoretic model for fading wireless channels where the
channel power gains remain invariant over a block (shorter than the coherence 
time of the channel) and may change from block to
block\cite{caitar99,bergal02}.\footnote{This 
model is appropriate for 
 many practical applications, and was considered, e.g., also 
in\cite{mosmur09}.}
Small-scale fading is incorporated into our model by allowing channel gains 
 at  different time slots and also gains of different links
to be conditionally independent for a given network state.  More formally, if
$\Ecal_{l}\not =  
\Ecal_{m}$ or $k \not = \ell$, then the channel
gain distributions are time-homogeneous and satisfy 
\begin{equation}
\label{eq:27}
\begin{split}
  \Prob \big\{ &h_{{l}}(k)\leq a_1 ,h_{m}(\ell)\leq a_2 \, \big| \, 
  \Xi(k)  = j, \Xi(\ell)  = i\big\}\\
  &= \Prob \{ h_{l}(k)\leq a_1  \, | \, \Xi(k)  = j\}\\
&\qquad\times \Prob \{h_{m}(\ell)\leq a_2 \, | \, \Xi(\ell)  = 
  i\},
\end{split}
\end{equation}
for all $a_1,a_2 \in \Rset_{\geq 0}$ and  all  $i,j \in\Bset$. Note that
in~(\ref{eq:27}) we do
not limit our attention to particular fading distributions. For example, our
model could use Rayleigh, Rician or Nakagami
distributions\cite{goldsm05}. Furthermore, individual links are allowed to
switch between different distribution classes.  
It is important to emphasize that, given~(\ref{eq:27}), the  network state process serves to describe expected channel
gains. Our framework allows for spatial correlations between channel
gains of individual links. It  also incorporates temporal correlations, as per the Markov
chain model of the network state process. For time-varying environments, our
model will be more meaningful than simply taking long term averages of link
gains; see also recent experimental studies documented in\cite{bjonet11}.

\subsection{Packet Loss}
\label{sec:packet-dropouts}
We will assume that each data packet is either
received perfectly or is completely lost and unavailable to the receiver;
cf.,\cite{mosmur09}. Transmission effects are modeled  via random packet
dropouts at the individual links of the network. We, thus, introduce the
binary stochastic communication success processes
$\{\gamma_m \}_{\Nset_0}$, $m \in\{1,2,\dots,M\}$, 
where
\begin{equation}
  \label{eq:3}
  \gamma_{m} (k) =
  \begin{cases}
    1&\text{if  at time $k$ transmission via $\Ecal_m$ is successful,}\\
  0& \text{otherwise.}
  \end{cases}
\end{equation}
The distributions of  the processes
$\{\gamma_m \}_{\Nset_0}$  are
determined by channel gains, bit-rates and power levels. To be more specific,
for each link  
$\Ecal_m = (S_m,\parent(S_m))$, $m\in\{1,\dots,M\}$, it holds that
\begin{equation}
  \label{eq:10}
  \begin{split}
    \Prob&\big\{\gamma_m (k)= 1 \,\big|\,h_m(k)=h, u_{m}(k)=u, b_m(k)=b \big\}\\
    &=f_{m}(hu,b),
  \end{split}
\end{equation}
where $u_{m}(k)$ denotes the power used by the radio power
amplifier of 
sensor $S_m$,  and $b_m(k)$  the corresponding bit-rate. In~(\ref{eq:10}),
$f_{m}\colon \Rset_{\geq 0}  \times B 
\rightarrow [0,1]$ where  $B$ is the set of allowable bit-rates. The function
$f_{m}$ is monotonically increasing in the first argument (the received
signal power) and
monotonically decreasing in the second argument (the bit-rate). Its specific
form depends on the modulation used by each node $S_m$, see, e.g.,
\cite{proaki95,queahl10}.\footnote{Fast retransmissions of individual links can
  be included into our 
  framework by simply replacing $f_m(hu,b)$ with $1-\big(1-f_m(hu,b)\big)^L$, where $L$ is the
  maximum number of retransmissions allowed by the protocol and delay constraint.}

\par In view of~\eqref{eq:10},  power and bit-rate control can be used to  counteract fading
effects; see \cite{panver07,queahl10,queahl12}. Throughout this work
we  allow transmission power levels and bit-rates to depend upon the channel
gains and the network state. We, thus, introduce the following standing assumption:

\begin{ass}
\label{ass:policies}
Power and bit-rate control laws are  of the form
\begin{equation}
  \label{eq:41}
  \begin{split}
    u_{m}(k) &= \kappa_m(\Xi(k),h_{1}(k), \dots,h_{M}(k)),\\
    b_m(k)  &= \eta_m(\Xi(k),h_{1}(k), \dots,h_{M}(k)),
  \end{split}
\end{equation}
where $\kappa_m$ and $\eta_m$, $m \in\{1,2,\dots,M\}$, are non-linear mappings. \hfs
\end{ass}
Particular cases of~\eqref{eq:41} include
the use of fixed power levels and bit-rates, 
fixed gain controllers with saturated outputs,
$u_{m}(k)=\mathrm{sat}\big(\mathcal{K}_m / h_{m}(k)\big)$; and also the
various 
power allocation policies studied in\cite{bergal02}.

\par A key feature of the network model presented above and the power and bit-rate controller
class considered is that, \emph{when conditioned 
  upon 
  the network state} $\Xi$, the  link transmission success processes are
independent in time and of each other. For further reference, we will denote the associated success
probabilities via
\begin{equation}
  \label{eq:43}
  {\phi}_{m|j}\eq \Prob\big\{\gamma_{m}(k) = 1 \,\big|\, \Xi(k) =j
  \big\},\,  m\in \{1,\dots,M\},j\in\Bset
\end{equation}
and note that for power and bit-rate control laws~\eqref{eq:41}, we have
\begin{equation}
  \label{eq:44}
  \begin{split}
    {\phi}_{m|j} =\E \big\{f_{m} &\big(
    h_{m}\cdot\kappa_m(\Xi ,h_{1},\dots, h_{M}),\\
    &\eta_m(\Xi,h_{1}, \dots,h_{M}) \big)\,\big|\,\Xi =j
    \big\}.
  \end{split}
\end{equation}
 Thus, for given control and bit-rate policies,
 calculating ${\phi}_{m|j}$ involves
simply taking expectation with respect to the conditional
distribution of the channel gains 
 given the network  state $\Xi=j$, see~(\ref{eq:27}).

\begin{figure}[t]
  \centering
  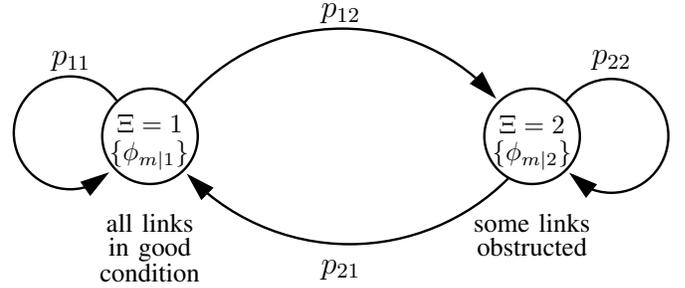
  \caption{Sensor network connectivity model with two states, correlated as per
    Assumption~\ref{ass:Markov}; see Example~\ref{ex:independent}.} 
  \label{fig:markov_tac12}
\end{figure}

\begin{ex}
\label{ex:independent}
 Suppose that there are only two network states: In state $\Xi=1$,
all $M$ links are in good condition (with low dropout probabilities). When $\Xi(k) = 2$,
some of the links are obstructed by a large object, thus, having  very
small  expected channel gains.  Transition between the two network states  is random and
obeys the Markov chain model~(\ref{eq:5}),  see Fig.~\ref{fig:markov_tac12}. \hfs
\end{ex}


 It is important to emphasize that the network state determines the
   distribution of the $M$  channel gains and  thereby the distribution of the link
   success probabilities in~(\ref{eq:43}). Despite Assumption~\ref{ass:Markov}, we
   do not require that the channel gains $\{h_m\}_{\Nset_0}$ or the dropout
   processes $\{\gamma_m\}_{\Nset_0}$ be Markovian.  
   Our model encompasses, as special cases, i.i.d.\
   transmissions\cite{sinsch04}, the Gilbert-Elliot Model \cite{gilber60,ellio63}, 
   Markovian models for the dropout processes
   \cite{xiexie08,huadey07,smisei03}, and channel gains described by a finite
   Markov chain\cite{bergal02}. 
 In the one-link case, and provided Assumption~\ref{ass:Markov} holds,  the 
   model presented is mathematically equivalent to the hidden Markov chain model
   of\cite{fleran04}. A key difference to\cite{fleran04}, however, is that the
   present model is physically motivated and allows one to incorporate the effect
   of power and 
   bit-rate control in an straightforward manner. In
   Section~\ref{sec:one-sensor-case} we will further examine  these
   relationships.

\begin{rem}
  The design of power and bit-rate control laws, and also how to estimate
  network states from dropout observations, lies outside the scope of the present
  work. References on the power and bit-rate control problem for state
  estimation with wireless links  include\cite{queahl10,queahl12}. To estimate network states, one
  can adapt hidden Markov chain estimation techniques, as described, e.g., in
  \cite{rabine89,barlim08} and used in\cite{bjonet11}. \hfs
\end{rem}

\section{State Estimation over a Sensor Network Tree with 
Packet Drop-outs} 
\label{sec:state-estim-over-1}
The purpose of the sensor network architecture considered is to estimate the
state 
of the system~\eqref{eq:1} centrally at the gateway by using the
measurements received from 
 the sensors $\{S_1,S_2,\dots,S_M\}$, see Fig.~\ref{fig:tree_tac12}.   As we have seen in
 Section~\ref{sec:transmission-effects}, fading channels 
will introduce random packet loss. 
 From an estimation point of view, it is
convenient to  
introduce the binary sensor-to-gateway connectivity processes
$\{\theta_{m}\}_{\Nset_0}$, 
 $m\in \{1,\dots,M\}$,  where
\begin{equation}
  \label{eq:7a}
  \theta_{m} (k) =
  \begin{cases}
    1&\text{if at time $k$ transmission via}\\
    &\qquad\text{$\path(S_m)$ is successful,}\\
  0& \text{otherwise}.
  \end{cases}
\end{equation}
\begin{rem}
  Since we assume that the network does not introduce
  any delays, 
we have 
\begin{equation*}
  \theta_{m} (k) = \prod_{\Ecal_i \in \edge(\path(S_m))}
  \gamma_{i} (k),\quad \forall m \in \{1,\dots,M\}.
\end{equation*}
Furthermore, the conditional distributions of
  $\{\theta_{m}\}$ given the network states can  be written in  terms of the
  individual link functions $ \phi_{i|j}$ introduced in~(\ref{eq:43}) as follows:
  \begin{equation}
\label{eq:20}
    \begin{split}
      \Prob&\{\theta_{m} (k)= 1 \,\big|\, \Xi(k)=j\} \\
      &= \prod_{\Ecal_i \in \edge(\path (S_m))}
      \Prob\{\gamma_{i} (k)= 1 \,\big|\, \Xi(k)=j\}\\
      &= \prod_{\Ecal_i \in \edge(\path (S_m))} \phi_{i|j}.
    \end{split}
  \end{equation}
Note, however, that if  $S_{m}\in\node(\path({S_{l}}))$, with $m\not = l$, then,
in general,
\begin{equation*}
  \begin{split}
    \Prob&\{\theta_{m} (k)=1, \theta_{l} (k)= 1 \,|\, \Xi(k)=j\}\\
    &\!\!\! \not =\Prob\{\theta_{m} (k)= 1 \,|\, \Xi(k)=j\}\times
    \Prob\{\theta_{l} (k)= 1 \,|\, \Xi(k)=j\},
  \end{split}
\end{equation*}
despite the fact that channel gain distributions satisfy~\eqref{eq:27}.\hfs
\end{rem}

We will assume that the packets transmitted from the sensors to the gateway
incorporate error detection coding, see, e.g., \cite{proaki95}, and that the
gateway 
knows, whether received packets
are error-free or not.
 Thus, the information available for state estimation at the gateway at time $k$
 is given by  
 \begin{equation}
   \label{eq:13}
   \Ical(k) = \Big\{ \{\theta_{1}\}_0^k,\dots, \{\theta_{M}\}_0^k,
   \{y\}_0^k \Big\},
 \end{equation}
where
\begin{equation*}
   y (k) \eq
   \begin{bmatrix}
    \theta_{1}(k) y_{1}(k)\\\theta_{2}(k) y_{2}(k) \\ \vdots
    \\\theta_{M}(k) y_{M}(k)
   \end{bmatrix}\!,\quad k \in \Nset_0.
\end{equation*}


\par With power and bit-rate control laws of the
form~\eqref{eq:41} 
and given 
the network fading model adopted, 
the dropout 
realizations in~\eqref{eq:13} do not convey information about the system state
$\{x\}_{\Nset_0}$. Since we have assumed that the network does not introduce any
delays, it turns out that state estimation in
the  
wireless sensor network configuration 
studied amounts 
to sampling the system~\eqref{eq:1} using the time-varying
(stochastic) 
observation 
matrix 
\begin{equation}
  \label{eq:42}
  C(k) \eq \begin{bmatrix}
    \theta_{1}(k) {C}_1 \\ \theta_{2}(k) {C}_2 \\  \vdots  \\
 \theta_{M}(k) {C}_M
   \end{bmatrix}\!,\quad k \in \Nset_0.
\end{equation}
Consequently, the
conditional  distribution of  $x(k)$
given $\Ical(k-1)$ 
 is Gaussian. The conditional mean of $x(k)$,
 \begin{equation*}
    \hat{x}(k | k-1) \eq\E \big\{x(k)  \,\big|\, \Ical(k-1)
    \big\} 
 \end{equation*}
and the associated  estimator   covariance matrix, 
 \begin{equation}
\label{eq:22}
{P}(k\,|\,k-1) \eq\E \big\{\epsilon(k)\epsilon(k)^T  \big\} 
\end{equation}
with
\begin{equation}
  \label{eq:21}
  \epsilon(k)\eq x(k) - \hat{x}(k | k-1),
\end{equation}
satisfy the Kalman filter recursions (see, e.g., \cite{andmoo79}):
\begin{equation}
  \label{eq:11}
  \begin{split}
    \hat{x}(k + 1|k) &= A(k)\hat{x}(k|k-1) \\
    &\qquad + K(k)
    \big(y(k) - C(k) \hat{x}(k|k-1)\big)\\
    P(k+1|k) & = A(k) P(k|k-1) A(k)^T +Q(k)\\
    &\qquad -  K(k) C(k) P(k|k-1) A(k)^T
  \end{split}
\end{equation}
where
\begin{equation*}
  \begin{split}
    R(k)&\eq \diag \big(R_1(k),R_2(k),\dots,R_M(k)\big),\\
    K(k)&\eq A(k)P(k | k-1) C(k)^T\\
    &\quad\cdot\big(C(k)P(k | k-1)C(k)^T +R(k) \big)^{-1},
  \end{split}
\end{equation*}
and with initial values  $P(0 |-1)=P_0$ and $\hat{x}(0 |-1)=x_0$.

\par It follows directly from~(\ref{eq:42}) that $C(k)$ takes one of $2^M$ possible
values. The probability distribution of $C(k)$ depends upon the current channel
gains, the bit-rates and the power levels used; see~(\ref{eq:44}) and (\ref{eq:20}).
Thus, $\{C\}_{\Nset_0}$ is a random process, the
 recursion~(\ref{eq:11}) is stochastic and the error covariance process
$\{P(k+1|k)\}_{k\in\Nset_0}$ is stochastic.

\par A key  difference of our approach when compared to that in
\cite{shicap10,shi09b}, is that we consider
packet dropouts. Thus, in case of open loop unstable systems~(\ref{eq:1}), the
estimator  covariance matrix will, in general not be 
stationary. The situation is akin to that encountered in the
context of state estimation over lossy communication links with constant dropout
probabilities; see, 
e.g., \cite{schsin07}, or also where  dropout processes are (semi-)Markovian; see, e.g.,
\cite{xiexie08,huadey07,smisei03,youfu11,censi11}. In 
the case under study in the present work, the plant model and transmission
success probabilities 
are time
varying, and channel gains are correlated between each other. Thus, the results
of the above articles 
cannot be applied directly. In the following section, we will take a closer look
at stochastic stability of the Kalman filter~(\ref{eq:11}).

\section{Stability Analysis for Markovian Network States}
\label{sec:stability-analysis}
If the system~(\ref{eq:1}) is unstable, then due to packet dropouts, the covariance matrix process
$\{P(k+1|k)\}_{k\in\Nset_0}$ in~(\ref{eq:11}) will, in general, not
converge to a fixed value and may, at times, diverge, thereby indicating poor
performance of the Kalman filter.  As shown in
\cite{schsin07}, this type of behaviour occurs even in the simplest scenario, where
only one sensor is used and dropout probabilities are i.i.d. 
We will next study
stability of the Kalman filter for the sensor network 
model presented in Section~\ref{sec:transmission-effects}. 
Towards that goal we adopt a stochastic stability notion, which captures the
fact that, with dropouts, the best one can hope for is that the estimator
covariance matrix   be bounded. More precisely, we will focus on the trace of
the covariance matrix, which by~(\ref{eq:22}) quantifies the mean square of the
estimation error, $  \tr P(k |k-1) = \E\big\{ \|\epsilon(k)\|^2\big\}$,
and adopt the
following definition, adapted from \cite{tarras76}:

\begin{defi}
\label{def:expbound}
  The   Kalman filter in~(\ref{eq:11}) is said to be exponentially
  bounded, if there exist finite constants $\alpha$ and $\beta$ and
  $\rho \in [0,1)$ such that:
  \begin{equation}
\label{eq:14}
  \E \big\{ \tr P(k |k-1) 
\big\} \leq \alpha \rho^k  + \beta ,\quad \forall k\in
  \Nset_0 .
\end{equation}

\end{defi}
\begin{rem}
Since $P(k|k-1)\succeq 0$, it directly follows that $\tr P(k |k-1)\geq \|P(k|k-1)\|=\sqrt{\max \eigs(P(k |k-1))}$ for
all $k\in \Nset_0$. Thus, exponential
 boundedness of the Kalman filter implies boundedness of $\E\{P(k|k-1)\}$ as studied,
 for example, in\cite{sinsch04,plabul09,liugol04} and also covariance stability, i.e., $\E
 \{\|P(k+1|k)\|\}<\infty$, for all $k\in 
  \Nset_0$, see, e.g.,\cite{xiexie08,huadey07}. Exponential
  boundedness  has also been used in \cite{quenes12a} for the analysis of a
  class of networked 
control systems with Markovian packet dropouts and non-vanishing disturbances.\hfs
\end{rem}

\par Our  analysis  makes use of the  process 
 $\{r\}_{\Nset_0}$, where  
\begin{equation}
  \label{eq:12}
  r(k) \eq
\begin{cases}
  1&\text{if $C(k)$ is full rank,}\\
  0& \text{otherwise.}
  \end{cases}
\end{equation} 
We also introduce
\begin{equation}
\label{eq:50}
  \nu_i \eq \Prob\{r(k)= 0 \,|\, \Xi(k-1)=i\}, \quad i\in\Bset,
\end{equation}
which denotes the  probability of $C(k)$ not being  full rank,
  given  $\Xi(k-1)=i$. Note that, by the law of total probabilities, we have
\begin{equation}
\label{eq:18}
  \begin{split}
  \nu_i &= \sum_{j\in\Bset} \Prob\{r(k)= 0 \,|\, \Xi(k-1)=i,\Xi(k)=j\}\\
  &\qquad\qquad\times\Prob\{\Xi(k)=j\,|\,\Xi(k-1)=i\}\\
  &= \sum_{j\in\Bset} p_{ij}\Prob\{r(k)= 0 \,|\, \Xi(k)=j\},\quad \forall i\in\Bset.
\end{split}
\end{equation}

Clearly, $r(k)$ is a (Boolean) function of the individual link success outcomes
$\gamma_m(k)$, $m\in\{1,\dots,M\}$ and thereby depends upon the channel gains, and
bit-rate and power control laws, see~(\ref{eq:44}). By the discussion in
Section~\ref{sec:transmission-effects}, it is easy to see that (provided
Assumption~\ref{ass:policies} holds),\ $r(k)$ is
temporarily independent, when
conditioned upon the network state $\Xi(k)$. Furthermore, and as with the
sensor-to-gateway connectivity processes $\theta_m(k)$,
the conditional distribution $\Prob\{r(k) \,|\, \Xi(k)\}$ can be written in terms of the
functions $\phi_{m|j}$ introduced in~\eqref{eq:43}, see also
Example~\ref{ex:morgan} included at the end of this section.

\par The following theorem gives a sufficient condition for  stability
of the Kalman filter used for state estimation over a sensor network with
Markovian channel states. 

\begin{thm}
\label{thm:multiple_sensors}
Suppose that Assumptions~\ref{ass:Markov} and~\ref{ass:policies} hold. If there exists $\rho \in [0,1)$ such that
  \begin{equation}
    \label{eq:18b}
   \max_{(i,k)\in\Bset\times \Nset_0}
\nu_i \|A(k)\|^2 \leq \rho,
  \end{equation}
then the  Kalman filter with the channel gain and power and bit-rate control
model  described in 
  Section~\ref{sec:transmission-effects} is exponentially bounded.
\end{thm}
\begin{proof}
See Appendix~\ref{sec:netw-with-mult-1}.
\end{proof}

Our result establishes a sufficient condition for  exponential boundedness  of
the  estimator covariance matrix when the channel gains
of the $M$ links
are governed by the radio environment model described in
Section~\ref{sec:transmission-effects}. The  condition requires that $C(k)$ be
full rank if all measurements are received and is stated in terms of a
bound which involves the spectral norm of the
system matrices $A(k)$, the transition probabilities of the channel state $\Xi$,
and the conditional probabilities of $\{r\}_{\Nset_0}$. The latter are determined by 
the individual conditional transmission success probabilities $\phi_{m|j}$, and
can therefore be influenced by designing the  power and bit-rate control
policies, see~\eqref{eq:44}. The 
situation investigated in the present work generalizes that studied for the simpler case of having independent channel gains
in our recent paper \cite{queahl12}.

\par Before turning our attention to a particular case, namely when the network has
only one sensor,  we will first give an example which illustrates how to evaluate $\Prob\{r(k)\,|\, \Xi(k)\}$.

\begin{figure}[t]
  \centering
  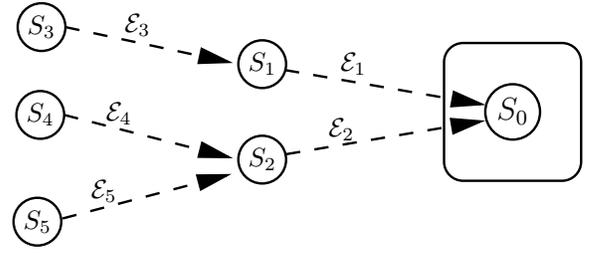
  \caption{Sensor network tree with $M=6$ nodes used in
    Example~\ref{ex:morgan}. Sensors $S_3$, $S_4$ and $S_5$ transmit their own
    measurements. At each time $k$,  $S_1$ transmits $y_1(k)$ and, if received, also
    $y_3(k)$. Likewise, $S_2$ transmits $y_2(k)$ and a subset of
    $\{y_4(k),y_5(k)\}$, depending upon the transmission outcomes of links
    $\mathcal{E}_4$ and  $\mathcal{E}_5.$}
  \label{fig:tree2_tac12}
\end{figure}

\begin{ex}
\label{ex:morgan}
  Consider the subgraph of the sensor network topology of
  Fig.~\ref{fig:tree_tac12} having vertices $\{S_0,S_1,\dots,S_5\}$, see Fig.~\ref{fig:tree2_tac12}. Suppose that for $C(k)$ to be 
  full rank (at least) three of the measurements $\{y_1(k),
  y_2(k),\dots,y_5(k)\}$ need to be received at the gateway. Then $r(k)=1$ if
  and only if 
\begin{equation*}
  \begin{bmatrix}
    \gamma_1(k)& \gamma_2(k) & \dots &\gamma_5(k)
  \end{bmatrix}^T\in
      \Jset,
    \end{equation*}
    where 
\begin{equation*}
 \Jset\eq\left\{
    \begin{bmatrix}
      0\\1\\0\\1\\1
    \end{bmatrix}\!,
\begin{bmatrix}
      0\\ 1\\1 \\1 \\1 
    \end{bmatrix}\!,
\begin{bmatrix}
      1\\ 1\\ 0\\ 0\\1 
    \end{bmatrix}\!,
\begin{bmatrix}
      1\\ 1\\ 0\\ 1\\ 0
    \end{bmatrix}\!,
\begin{bmatrix}
     1 \\ 1\\ 0\\1 \\ 1
    \end{bmatrix}\!,
\begin{bmatrix}
     1 \\1 \\1 \\0 \\ 0
    \end{bmatrix}\!,
\begin{bmatrix}
     1 \\1 \\1 \\0 \\ 1
    \end{bmatrix}\!,
\begin{bmatrix}
     1 \\1 \\ 1\\ 1\\ 0
    \end{bmatrix}\!,
\begin{bmatrix}
     1 \\1 \\ 1\\1 \\ 1
    \end{bmatrix}
\right\}\!.
\end{equation*}
If Assumption~\ref{ass:policies} holds, then, as noted in Section~\ref{sec:transmission-effects}, the link transmission
success processes $\gamma_m$ are conditionally independent. Thus, the conditional
      probabilities of $C$ being full rank can be obtained from $\Jset$ as follows:
      \begin{equation*}
        \begin{split}
          &\Prob\{r= 1 \,|\, \Xi=j\}
          =(1-{\phi}_{ 1|j}) {\phi}_{ 2|j} (1-{\phi}_{ 3|j}){\phi}_{ 4|j} {\phi}_{ 5|j}\\
          &+ (1-{\phi}_{ 1|j}) {\phi}_{ 2|j} {\phi}_{ 3|j} {\phi}_{ 4|j} {\phi}_{ 5|j}
          + {\phi}_{ 1|j} {\phi}_{ 2|j} (1-{\phi}_{ 3|j})  (1-{\phi}_{ 4|j}){\phi}_{ 5|j}\\
          &+{\phi}_{ 1|j} {\phi}_{ 2|j} (1-{\phi}_{ 3|j}){\phi}_{ 4|j}
          (1-{\phi}_{ 5|j})+{\phi}_{ 1|j} {\phi}_{ 2|j} (1-{\phi}_{ 3|j})
    {\phi}_{ 4|j} {\phi}_{ 5|j}\\
          &+{\phi}_{ 1|j} {\phi}_{ 2|j} {\phi}_{ 3|j}
    (1-{\phi}_{ 4|j}) (1-{\phi}_{ 5|j})+{\phi}_{ 1|j} {\phi}_{ 2|j} {\phi}_{ 3|j}
    (1-{\phi}_{ 4|j}) {\phi}_{ 5|j}\\
    &+{\phi}_{ 1|j} {\phi}_{ 2|j} {\phi}_{ 3|j}
    {\phi}_{ 4|j} (1-{\phi}_{ 5|j})+{\phi}_{ 1|j} {\phi}_{ 2|j} {\phi}_{ 3|j}
    {\phi}_{ 4|j} {\phi}_{ 5|j},
  \end{split}
\end{equation*}
for all $j\in\Bset$.\hfs
\end{ex}

\section{The One-sensor Case for LTI Systems}
  \label{sec:one-sensor-case}
Here, we focus on a particular instance of the sensor network model of interest,
namely, where there is only one sensor and one edge, and the system~(\ref{eq:1})
is LTI. In this case, it is easy to see that $C(k) =\gamma_{1}(k) C_1$ and the
estimator covariance matrix in~\eqref{eq:11} satisfies
\begin{equation}
\begin{split}
  \label{eq:9b}
   P(k+1&|k)  = A  P(k|k-1) A^T +Q -   K(k) C_1 P(k|k-1) A^T\\
   K(k)&=\gamma_{1}(k)A P(k|k-1) C_1^T\!\big(C_1 P(k|k-1)C_1^T +R \big)^{-1}.
 \end{split}
\end{equation}
Theorem~\ref{thm:multiple_sensors} can be directly applied to this setup, yielding the following result:

\begin{coro}
\label{thm:one-sensor-case}
  Consider the model introduced  in 
  Section~\ref{sec:transmission-effects} with $M=1$.  Suppose that
  Assumptions~\ref{ass:Markov} and~\ref{ass:policies} hold, and that $C_1$
  in~\eqref{eq:2a} is full   
  rank. If there exists $\rho \in [0,1)$ such that
  \begin{equation}
    \label{eq:2}
   \|A\|^2  \sum_{j\in\Bset}p_{ij}(1-\phi_{1|j}) \leq \rho,\quad \forall i\in\Bset,
  \end{equation}
then the Kalman filter is exponentially bounded.
\end{coro}
\begin{proof}
Immediate from Theorem~\ref{thm:multiple_sensors}, since $r(k)=\gamma_1(k)$.  
\end{proof}

\par Kalman filtering with a single sensor-link and  Markovian 
packet dropouts was investigated in
\cite{xiexie08,huadey07,youfu11}. Markovian dropouts correspond to the particular case of the setup
considered in Corollary~\ref{thm:one-sensor-case}, namely, where
$  \Bset=\{1,2\}$, $\phi_{1|1}=1$, and $\phi_{1|2}=0$, see
Fig.~\ref{fig:markov_tac12}.  Direct calculations give that
 with these
parameters the
sufficient condition for stochastic stability~(\ref{eq:2}) reduces
to
\begin{equation*}
  \|A\|^2 \max (p_{12},p_{22})\leq \rho<1,
\end{equation*}
thus, Corollary~\ref{thm:one-sensor-case} 
 becomes
 akin to Theorem 3  in 
\cite{xiexie08}. In a similar manner, it can be shown that the hidden Markov
model of\cite{fleran04} (which generalizes the Gilbert loss model\cite{gilber60}  and a
fixed-length burst loss description) can be recovered by setting
$\phi_{1|j}\in\{0,1\}$ for all $j\in \Bset$ in the hypotheses of
Corollary~\ref{thm:one-sensor-case}.  Vice-versa,  the model considered in
Corollary~\ref{thm:one-sensor-case}, 
can be stated in terms of the hidden Markov chain model
of\cite{fleran04} by considering the aggregated state process
$\{(\gamma_1,\Xi)\}_{\Nset_0}$. 


\par Our model 
can be further simplified by allowing for   only one network state,
i.e., by setting  $|\Bset|=1$ in~(\ref{eq:15}). In this case, we obtain
 a system with i.i.d.\
dropouts having transmission success probabilities
$\Prob\{\gamma_1(k)=1\}=\phi_{1|1}$, see~(\ref{eq:43}). Whilst this situation will not often be
encountered in practice, it certainly is of significant system-theoretic
importance, and has been extensively studied; see, e.g.,
\cite{schsin07,sinsch04,plabul09,censi11,rohmar10,epsshi08,karsin12,shieps10}.
With i.i.d.\ dropouts, the 
 condition~(\ref{eq:2}) becomes
$$\Prob\{\gamma_1(k)=0\}\|A\|^2\leq\rho<1,$$ thereby
 resembling various
conditions which have been reported in the literature; see
\cite{schsin07}.\footnote{The i.i.d.\ dropout case can also be regarded as a
  special instance of Corollary~\ref{thm:one-sensor-case}, where 
$  \Bset=\{1,2\}$, $\phi_{1|1}=1$, $\phi_{1|2}=0$, $p_{21}=p_{11}$, and
$p_{12}=p_{22}$ is the dropout probability.}

\section{Network States with Arbitrary Holding Times}
\label{sec:alternative-model}
In Sections~\ref{sec:stability-analysis} and~\ref{sec:one-sensor-case}  we
assumed that the network state may change at every 
instant $k\in\Nset_0$, see Assumption~\ref{ass:Markov}. This model serves to
describe situations where the radio environment changes relatively fast. 
%
We will next present a more general sensor network connectivity model. It
allows one to impose minimum holding times on the network states and is thereby
especially tailored for situations where the radio environment changes slowly.

\begin{figure}[t]
  \centering
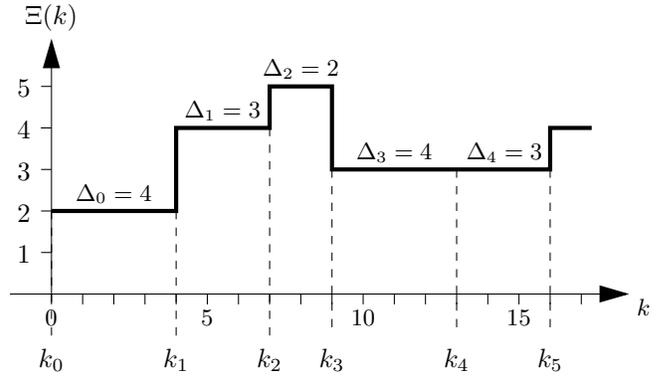
  \caption{Semi-Markov network state model as per
    Assumption~\ref{ass:dwell}.} 
  \label{fig:semi_model}
\end{figure}

\subsection{Semi-Markov Model}
\label{sec:semi-markov-model}
We will allow the times
between switches to follow an arbitrary probability
distribution. At the 
instants of transitions,  henceforth denoted by the
ordered set  
\begin{equation}
  \label{eq:8}
  \Kset \eq \{k_\ell\}_{\ell \in \Nset_0}\subseteq \Nset_0, \quad k_0=0,
\end{equation}
the process $\{\Xi\}_{\Kset_0}$ behaves like a Markov chain.  Thus, our model fits into the
   semi-Markovian framework considered, e.g.,
   in\cite{cinlar69,howard71,kleinr76,barlim08,censi11}. To 
   be more
  precise, we introduce the following assumption:
\begin{ass}
\label{ass:dwell}
The \emph{holding  times},
\begin{equation}
\label{eq:6}
\Delta_\ell\eq k_{\ell+1}-k_\ell\in\Nset, \quad \ell \in
\Nset_0,
\end{equation}
and the following transition $\Xi(k_{\ell+1})$ are conditionally
independent given (and depend only on)  the current state, $\Xi(k_{\ell})$, 
i.e., we have
%
\begin{equation}
\label{eq:19}
\begin{split}
  \Prob\{\Xi&(k_{\ell+1})=j,\Delta_\ell=\delta
  \,|\,\Xi(k_0),\Xi(k_1)\dots,\Xi(k_{\ell-1}),\\ 
  &\Xi(k_\ell)=i, k_0,\dots,k_\ell
  \}
  =q_{ij} \psi_{i}(\delta),
\end{split}
  \end{equation}
for all $i,j\in \Bset$, $\delta\in\Nset$, where
\begin{equation}
\label{eq:25}
  \begin{split}
    \psi_{i}(\delta)&\eq \Prob\{\Delta_\ell = \delta \,|\,\Xi(k_\ell)=i \}\\
    q_{ij}&\eq\Prob\{\Xi(k_{\ell+1})=j\,|\,\Xi(k_{\ell})=i\},
  \end{split}
\end{equation}
are the conditional distribution of  the holding times and the transition
probabilities of the embedded Markov chain $\{\Xi(k_{\ell})\}_{\ell \in
  \Nset_0}$, respectively.\hfs
\end{ass}

Note that if in~\eqref{eq:6} we have $\Delta_\ell \geq 2$, for some $\ell\in
\Nset_0$, then
\begin{equation}
  \label{eq:16}
  \Xi(k_\ell)=\Xi(k_\ell+1) = \dots =
\Xi(k_\ell+\Delta_\ell-1).
\end{equation}
Thus, the \emph{renewal process}
$\{(\Xi(k_\ell),k_\ell)\}_{\ell\in\Nset_0}$   describes the network
  state trajectory \emph{at all times} $k\in\Nset_0$.
It is worth emphasizing that, unless $q_{ii}=0$ for all $i\in\Bset$, \eqref{eq:19} allows for \emph{virtual
transitions}, i.e., where $\Xi(k_{\ell+1})=\Xi(k_\ell)$, see
Figs.~\ref{fig:semi_model} and~\ref{fig:dwell_model2}. The above transition
model generalizes the Markov model in Assumption~\ref{ass:Markov}, see also Fig.~\ref{fig:markov_tac12}, by allowing
one to assign holding time distributions. \emph{Inter-alia}, this serves to
capture 
situations where the environment changes slowly, in relation to the
sampling frequency of the system~\eqref{eq:1}, see~\eqref{eq:16}.

\begin{figure}[t]
  \centering
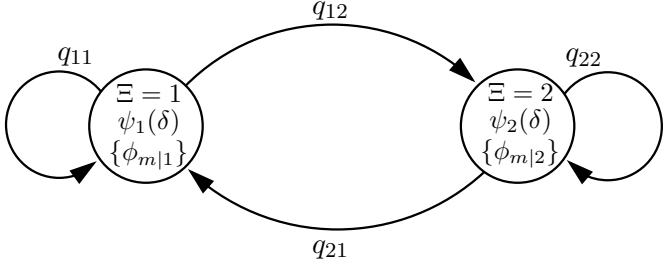
  \caption{Semi-Markov sensor network connectivity model as per
    Assumption~\ref{ass:dwell}, where $\Bset = \{1,2\}$.} 
  \label{fig:dwell_model2}
\end{figure}

\begin{ex}
  Consider the scheme in Fig.~\ref{fig:truck_tac12} and assume that the
  position $\Xi(k)=1$ is the robot ``home'' position (e.g., where its batteries
  are charged). The robot trajectory could, for example, be described by the model in
  Assumption~\ref{ass:dwell} with transition probabilities
  \begin{alignat*}{4}
    q_{11}&=0.8,&\quad q_{12}&=0.1,&\quad q_{13}&=0,&\quad q_{14}&=0.1,\\
    q_{21}&=0.5,&\quad q_{22}&=0.3,&\quad q_{23}&=0.1,&\quad q_{24}&=0.1,\\
    q_{31}&=0,&\quad q_{32}&=0.2,&\quad q_{33}&=0.5,&\quad q_{34}&=0.3,\\
    q_{41}&=0.6,&\quad q_{42}&=0.2,&\quad q_{43}&=0.1,&\quad q_{44}&=0.1, 
  \end{alignat*} 
and holding time distributions
\begin{equation*}
  \begin{split}
    \psi_1(\delta) &=
    \begin{cases}
       1&\text{if $\delta =30$,}\\
       0&\text{if $\delta \not =30$.}
    \end{cases}\\
     \psi_2(\delta) &= \psi_3(\delta)= 1/8,\quad\text{if $\delta \in\{1,2,\dots,8\}$},\\ 
    \psi_4(\delta) &= 1/6,\quad\text{if $\delta \in\{5,6,\dots,10\}$}.
  \end{split}
\end{equation*}
In the following, we will study stochastic stability of the Kalman filter~(\ref{eq:11}) for 
 sensor network setups  of this type. \hfs
\end{ex}

\subsection{Stability Analysis}
\label{sec:stability-analysis-1}
To characterize the system behaviour with the semi-Markov model of
Assumption~\ref{ass:dwell}, it is convenient to  introduce the \emph{transition matrix}
$\Phi(\ell,k)$, see\cite{andmoo79}, via
\begin{equation*}
  \begin{split}
    \Phi(\ell,k)&=A(\ell-1)A(\ell-2) \dots A(k),\quad \ell>k\\
    \Phi(\ell,\ell)&=I_n,
  \end{split}
\end{equation*}
so that
\begin{equation*}
  x(k+i)=\Phi(k+i,k)x(k)+\sum_{\ell=0}^{i-1}\Phi(k+i,k+\ell+1)w(k+\ell),
\end{equation*}
expression which follows directly from~(\ref{eq:1}). We also denote the 
observability matrix of order $t\in\Nset_0$ with initial step $k$, via
$\mathcal{O}(k,k)=C(k)$, and
\begin{equation*}
  \mathcal{O}(k+t,k) =
  \begin{bmatrix}
    C(k)\Phi(k,k)\\
    C(k+1)\Phi(k+1,k)\\
    \vdots\\
    C(k+t)\Phi(k+t,k)
  \end{bmatrix},\quad t \in \Nset.
\end{equation*}
 Our analysis  makes use of the  process
 $\{\varrho(k_\ell)\}_{k_\ell\in\Kset}$, where  
\begin{equation}
  \label{eq:12b}
  \varrho(k_\ell) \eq
\begin{cases}
  1&\text{if $\mathcal{O}(k_\ell+\Delta_\ell-1,k_\ell)$ is full rank,}\\
  0& \text{otherwise.}
  \end{cases}
\end{equation} 
Clearly, $\varrho(k_\ell)$ is a function of the individual link success outcomes
$\{\gamma_m(k)\}_{k_\ell}^{k_{\ell+1}-1}$, $m\in\{1,\dots,M\}$ and of
$\{A(k)\}_{k_\ell}^{k_{\ell+1}-2}$. Akin to what was done in
Section~\ref{sec:stability-analysis}, see~(\ref{eq:50}), we define
\begin{equation}
  \label{eq:33}
    \begin{split}
      \mu_i&(k_\ell,\delta)\eq \Prob\{\varrho(k_\ell) =0\,|\,\Xi(k_\ell-1)=i,
      \Delta_\ell=\delta\} \\
      &=\sum_{j\in\Bset}\Prob\{\varrho(k_\ell) =0\,|\,\Xi(k_\ell-1)=i,
      \Delta_\ell=\delta,\Xi(k_\ell)=j\}\\
      &\qquad \times\Prob\{\Xi(k_\ell)=j\,|\,\Xi(k_\ell-1)=i,\Delta_\ell=\delta\}\\
      &= \sum_{j\in\Bset}q_{ij} \Prob\{\varrho(k_\ell)
      =0\,|\,\Xi(k_\ell)=j, \Delta_\ell=\delta\}, \quad i\in \Bset,
    \end{split}
   \end{equation}
which denotes the  probability of  $\mathcal{O}(k_\ell+\delta-1,k_\ell)$ not being
full rank, conditioned on $\Xi(k_\ell-1)=i$. The following result establishes
sufficient conditions for exponential boundedness of the Kalman
filter~(\ref{eq:11}) for cases where 
$\{\Delta_\ell\}_{\ell\in\Nset_0}$ has bounded support.

\begin{thm}
\label{thm:semimarkov}
  Suppose that Assumptions~\ref{ass:policies} and~\ref{ass:dwell} hold and that
there exists a finite value
$\sigma$ such that
$\Delta_\ell\leq \sigma$, for all $\ell \in\Nset_0$.\footnote{Since our model
  allows for virtual transitions, we do not impose that the network state
  changes, at most, every $\sigma$ instants.} If
 there exists $\rho \in [0,1)$ such that
  \begin{equation}
    \label{eq:18c}
   \max_{(i,k_\ell)\in\Bset\times \Kset}
\sum_{\delta=1}^\sigma \mu_i(k_\ell,\delta) \sum_{j\in\Bset} \psi_{j}(\delta)q_{ij}
     \|\Phi(k_\ell+\delta,k_\ell)\|^2 \leq \rho^\sigma,
  \end{equation}
then the Kalman filter in~(\ref{eq:11}) is exponentially bounded.
\end{thm}
\begin{proof}
  See Appendix~\ref{sec:proof-theor-refthm:s}.
\end{proof}
To elucidate the condition~(\ref{eq:18c}), it is convenient to recall the
definitions~(\ref{eq:25}) and~(\ref{eq:33}) and note that for a
given holding time $\Delta_\ell=\delta\in\{1,2,\dots,\sigma\}$, we have
\begin{equation*}
  \begin{split}
    \mu_i(k_\ell,\delta) &\sum_{j\in\Bset}\psi_{j}(\delta)q_{ij}\\
     &=   \Prob\{\varrho(k_\ell)
     =0\,|\,\Xi(k_\ell-1)=i, \Delta_\ell=\delta\}\\ 
     &\quad\times\sum_{j\in\Bset} \Prob\{\Delta_\ell = \delta \,|\,\Xi(k_\ell)=j \}\\ 
     &\quad\times \Prob\{\Xi(k_{\ell})=j\,|\,\Xi(k_{\ell}-1)=i\} \\
    &= \Prob\{\varrho(k_\ell) =0\,|\,\Xi(k_\ell-1)=i,
      \Delta_\ell=\delta\}\\ 
     &\quad\times \Prob\{\Delta_\ell = \delta \,|\,\Xi(k_\ell-1)=i \} . 
  \end{split}
\end{equation*}
Thus, the left-hand-side of~(\ref{eq:18c}) extends the left-hand-side
of~(\ref{eq:18b}) for use in the semi-Markov model by
averaging non-full-rank observation outcomes over finite horizons.    Here it is worth noting
 that the Markovian network model  of
Assumption~\ref{ass:Markov} corresponds to the special instance of the
model in Assumption~\ref{ass:dwell}, wherein $p_{ij}=q_{ij}$ for all $i,j\in\Bset$, and
$\Delta_\ell=1$ for all 
$\ell \in \Nset_0$, so that $\psi_{j}(1)=1$ for all $j\in\Bset$,  $\sigma =1$
and $\Kset=\Nset_0$. In addition, we have 
$\Phi(k+1,k)=A(k)$,
\begin{equation*}
  \begin{split}
    \mu_i(k,1)&=\sum_{j\in\Bset}p_{ij} \Prob\{\varrho(k) =0\,|\,\Xi(k)=j,
    \Delta_\ell=1\}\\
    &=\sum_{j\in\Bset}p_{ij} \Prob\{r(k) =0\,|\,\Xi(k)=j\}=\nu_i,
  \end{split}
\end{equation*}
and $\sum_{j\in\Bset} q_{ij}=1$ for all $i\in \Bset$. Consequently, it is easy
to see that, in 
this case, the
condition~\eqref{eq:18c}  condenses into~\eqref{eq:18b} and we recover the
result established in Theorem~\ref{thm:multiple_sensors}. 
\begin{rem}
 In contrast to Theorem~\ref{thm:multiple_sensors},  Theorem~\ref{thm:semimarkov} 
does not require  that the matrix $C(k)$ be full rank with non-zero probability in order to establish exponential
 boundedness of the Kalman filter. \emph{Inter-alia}, Theorem~\ref{thm:semimarkov} requires
that when no dropouts occur the system (\ref{eq:1})--(\ref{eq:2a}) be observable over horizons of length $\sigma$.  \hfs 
\end{rem}


\subsection{Example}
\label{sec:example}
To illustrate the use of Theorem~\ref{thm:semimarkov}, we examine a simple LTI plant model with  $M=1$
sensor and where 
\begin{equation}
  \label{eq:23}
  A=
  \begin{bmatrix}
    1.25 & 0\\ 1 &1.1
  \end{bmatrix},
\quad C_1=
\begin{bmatrix}
  1&1
\end{bmatrix},
\end{equation}
taken from
\cite{sinsch04}. The connectivity of $S_1$ to the gateway is described by two possible
configurations, $\Xi(k)\in\Bset=\{1,2\}$, which obey Assumption~\ref{ass:dwell},
with  
\begin{equation*}
  \label{eq:36}
  \begin{split}
    \psi_1(\delta)&= 1/5, \quad
    \text{if $\delta \in\{1,2,\dots,5\}$},\\ 
  \psi_2(\delta)&=  1/7, \quad
  \text{if $\delta \in\{1,2,\dots,7\}$},
\end{split}
\end{equation*}
see Fig~\ref{fig:dwell_model2}. Power and bit-rate control laws are of the form~\eqref{eq:41}.
\par For one-sensor LTI systems, we have $C(k)=\gamma_1(k)C_1$, and
$\Phi(k+t,k)=A^t$ (which is obtained by setting  $A(k)=A$),
for all 
$k,t\in\Nset_0$, thus,
\begin{equation*}
  \label{eq:47}
  \mathcal{O}(k_\ell+\delta-1,k_\ell)=
  \begin{bmatrix}
    \gamma_1(k_\ell)C_1\\
    \gamma_1(k_\ell+1)C_1A\\
    \vdots\\
    \gamma_1(k_\ell+\delta-1)C_1A^{\delta-1}
  \end{bmatrix}.
\end{equation*}
It is easy to verify that for the system matrices given  in~\eqref{eq:23}, the matrix
\begin{equation*}
  \begin{bmatrix}
    C_1\\C_1A^r
  \end{bmatrix}
\end{equation*}
is invertible, for all $r\in\{1,\dots,6\}$. Therefore,
$\varrho(k_\ell)=0$, if and only if
\begin{equation*}
  \sum_{t=0}^{\Delta_\ell-1} \gamma_1(k_\ell+t)\leq 1.
\end{equation*}
and the conditional probabilities in~(\ref{eq:33}) become 
\begin{equation*}
  \begin{split}
    \mu_i&(k_\ell,\delta)\\
    &= \sum_{j=1}^2 q_{ij}  \Prob\Bigg\{
    \sum_{t=0}^{\Delta_\ell-1}
    \gamma_1(k_\ell+t)=0\,\bigg|\,\Xi(k_\ell)=j,\Delta_\ell=\delta\Bigg\}\\
    &\, +\sum_{j=1}^2 q_{ij} \Prob\Bigg\{\sum_{t=0}^{\Delta_\ell-1}
    \gamma_1(k_\ell+t)=1\,\bigg|\,\Xi(k_\ell)=j,\Delta_\ell=\delta\Bigg\}.
  \end{split}
\end{equation*}
Thus, for $\delta =1$ we obtain
\begin{equation*}
  \begin{split}
    \mu_i(k_\ell,1)&= \sum_{j=1}^2 q_{ij} \Prob\big\{
    \gamma_1(k_\ell)\in\{0,1\}\,|\,\Xi(k_\ell)=j,\Delta_\ell=1\big\}\\
    &=\sum_{j=1}^2
    q_{ij}=1,
  \end{split}
\end{equation*}
whereas, for $\delta \geq 2$,
\begin{equation}
\label{eq:24}
    \mu_i(k_\ell,\delta)= \sum_{j=1}^2
    q_{ij}\big((1-\phi_{1|j})^\delta+\delta(1-\phi_{1|j})^{\delta-1}\phi_{1|j}\big),
\end{equation}
see~(\ref{eq:43}). Given the above, Theorem~\ref{thm:semimarkov} establishes
that the Kalman filter~(\ref{eq:11}) is exponentially bounded if there exists
$\rho\in[0,1)$, such that
\begin{equation*}
  \begin{split}
    &\max_{i\in\{1,2\}} \sum_{\delta=1}^7\|A^\delta\|^2 \mu_i(k_\ell,\delta)
\Big(\psi_{1}(\delta)q_{i1}+\psi_{2}(\delta)q_{i2} \Big) \\
&=\max_{i\in\{1,2\}} \bigg(\frac{ q_{i1}}{5}+\frac{q_{i2}}{7} \bigg) \bigg( 
\|A\|^2\\
&\quad+
\sum_{\delta=2}^5\|A^\delta\|^2 \mu_i(k_\ell,\delta) \bigg)
 +\frac{q_{i2}}{7}\sum_{\delta=6}^7\|A^\delta\|^2 \mu_i(k_\ell,\delta)\leq \rho^7,
  \end{split}
\end{equation*}
where $\mu_i(k_\ell,\delta)$ are given in~(\ref{eq:24}).

\section{Conclusions}
\label{sec:conclusions-1}
In this work we have studied  Kalman filtering 
for state estimation over a wireless sensor network.
Since the radio links
between the nodes are fading, even if alleviated by power and bit-rate control, packet drops
may occur. To model different radio connectivity configurations of the environment,
we have introduced  a network state
process. Through the use of stochastic
stability methods,  we have derived sufficient conditions for the Kalman filter
covariance matrix to be exponentially bounded  when the underlying
network state is described by a  (semi-)Markov chain. Under this assumption, channel gains will be correlated over space
and time,
which is a suitable model when considering shadow fading. 
 In  particular cases, the sufficient condition obtained reduce to  stability
results previously documented in the literature. 

\par Future
work includes complementing the sufficient conditions for exponential boundedness of the
estimator presented with necessary ones, and using the results for the design of power
and bit-rate control and re-routing strategies. Also of interest is extending the
fading network model
proposed to more general topologies  and study its use  for the
analysis and design of closed loop networked control 
system architectures.

\bibliography{/Users/daniel/Dropbox/dquevedo}

\begin{thebibliography}{10}
\providecommand{\url}[1]{#1}
\csname url@rmstyle\endcsname
\providecommand{\newblock}{\relax}
\providecommand{\bibinfo}[2]{#2}
\providecommand\BIBentrySTDinterwordspacing{\spaceskip=0pt\relax}
\providecommand\BIBentryALTinterwordstretchfactor{4}
\providecommand\BIBentryALTinterwordspacing{\spaceskip=\fontdimen2\font plus
\BIBentryALTinterwordstretchfactor\fontdimen3\font minus
  \fontdimen4\font\relax}
\providecommand\BIBforeignlanguage[2]{{%
\expandafter\ifx\csname l@#1\endcsname\relax
\typeout{** WARNING: IEEEtran.bst: No hyphenation pattern has been}%
\typeout{** loaded for the language `#1'. Using the pattern for}%
\typeout{** the default language instead.}%
\else
\language=\csname l@#1\endcsname
\fi
#2}}

\bibitem{queahl11b}
D.~E. Quevedo, A.~Ahl\'{e}n, and K.~H. Johansson, ``Stability of state
  estimation over sensor networks with {M}arkovian fading channels,'' in
  \emph{Proc.~{IFAC} World Congr.}, 2011.

\bibitem{chejoh11}
J.~Chen, K.~H. Johansson, S.~Olariu, I.~C. Paschalidis, and Stokmenovic,
  ``Guest editorial special issue on wireless sensor and actuator networks,''
  \emph{{IEEE} Trans. Automat. Contr.}, vol.~56, no.~11, pp. 2244--2246, Oct.
  2011.

\bibitem{aluinn11}
R.~Alur, A.~D'Innocenzo, K.~H. Johansson, G.~J. Pappas, and G.~Weiss,
  ``Compositional modeling and analysis of multi-hop control networks,''
  \emph{{IEEE} Trans. Automat. Contr.}, vol.~56, no.~10, pp. 2345--2357, Oct.
  2011.

\bibitem{bjonet11}
M.~Bj\"{o}rkbom, S.~Nethi, L.~M. Eriksson, and R.~J\"{a}ntti, ``Wireless
  control system design and co-simulation,'' \emph{Contr. Eng. Pract.},
  vol.~19, no.~9, pp. 1075--1086, Sept. 2011.

\bibitem{ilymah04}
M.~Ilyas, I.~Mahgoub, and L.~Kelly, \emph{Handbook of Sensor Networks: Compact
  Wireless and Wired Sensing Systems}.\hskip 1em plus 0.5em minus 0.4em\relax
  Boca Raton, FL, USA: {CRC}-Press, Inc, 2004.

\bibitem{shezha07}
X.~Shen, Q.~Zhang, and R.~{Caiming Qiu}, ``Wireless sensor networking [guest
  ed.],'' \emph{{IEEE} Wireless Commun.}, vol.~14, no.~6, pp. 4--5, Dec. 2007.

\bibitem{goldsm05}
A.~Goldsmith, \emph{Wireless Communications}.\hskip 1em plus 0.5em minus
  0.4em\relax Cambridge University Press, 2005.

\bibitem{proaki95}
J.~G. Proakis, \emph{Digital Communications}, 3rd~ed.\hskip 1em plus 0.5em
  minus 0.4em\relax New York, N.Y.: McGraw-Hill, 1995.

\bibitem{caitar99}
G.~Caire, G.~Taricco, and E.~Biglieri, ``Optimum power control over fading
  channels,'' \emph{{IEEE} Trans. Inform. Theory}, vol.~45, no.~5, pp.
  1468--1489, July 1999.

\bibitem{bergal02}
R.~A. Berry and R.~G. Gallager, ``Communication over fading channels with delay
  constraints,'' \emph{{IEEE} Trans. Inform. Theory}, vol.~48, no.~5, pp.
  1135--1149, May 2002.

\bibitem{queahl10}
D.~E. Quevedo, A.~Ahl\'{e}n, and J.~{\O stergaard}, ``Energy efficient state
  estimation with wireless sensors through the use of predictive power control
  and coding,'' \emph{{IEEE} Trans. Signal Processing}, vol.~58, no.~9, pp.
  4811--4823, Sept. 2010.

\bibitem{queahl12}
D.~E. Quevedo, A.~Ahl\'{e}n, A.~S. Leong, and S.~Dey, ``On {K}alman filtering
  over fading wireless channels with controlled transmission powers,''
  \emph{Automatica}, vol.~48, no.~7, pp. 1306--1316, July 2012.

\bibitem{shicap10}
L.~Shi, A.~Capponi, K.~H. Johansson, and R.~M. Murray, ``Resource optimization
  in a wireless sensor network with guaranteed estimator performance,''
  \emph{{IET} Control Theory Appl.}, vol.~4, no.~5, pp. 710--723, 2010.

\bibitem{shi09b}
L.~Shi, ``Kalman filtering over graphs: Theory and applications,'' \emph{{IEEE}
  Trans. Automat. Contr.}, vol.~54, no.~9, pp. 2230--2234, Sept. 2009.

\bibitem{gupdan09}
V.~Gupta, A.~F. Dana, J.~P. Hespanha, R.~M. Murray, and B.~Hassibi, ``Data
  transmission over networks for estimation and control,'' \emph{{IEEE} Trans.
  Automat. Contr.}, vol.~54, no.~8, pp. 1807--1819, Aug. 2009.

\bibitem{chisch11}
A.~Chiuso and L.~Schenato, ``Information fusion strategies and performance
  bounds in packet-drop networks,'' \emph{Automatica}, vol.~47, pp. 1304--1316,
  July 2011.

\bibitem{mogar11}
Y.~Mo, E.~Garone, A.~Casavola, and B.~Sinopoli, ``Stochastic sensor scheduling
  for energy constrained estimation in multi-hop wireless sensor networks,''
  \emph{{IEEE} Trans. Automat. Contr.}, vol.~56, no.~10, pp. 2489--2495, Oct.
  2011.

\bibitem{sinsch04}
B.~Sinopoli, L.~Schenato, M.~Franceschetti, K.~Poolla, M.~I. Jordan, and S.~S.
  Sastry, ``Kalman filtering with intermittent observations,'' \emph{{IEEE}
  Trans. Automat. Contr.}, vol.~49, no.~9, pp. 1453--1464, Sept. 2004.

\bibitem{plabul09}
K.~Plarre and F.~Bullo, ``On {K}alman filtering for detectable systems with
  intermittent observations,'' \emph{{IEEE} Trans. Automat. Contr.}, vol.~54,
  no.~2, pp. 386--390, Feb. 2009.

\bibitem{karsin12}
S.~Kar, B.~Sinopoli, and J.~M.~F. Moura, ``Kalman filtering with intermittent
  observations: Weak convergence to a stationary distribution,'' \emph{{IEEE}
  Trans. Automat. Contr.}, vol.~57, no.~2, pp. 405--420, Feb. 2012.

\bibitem{liugol04}
X.~Liu and A.~Goldsmith, ``Kalman filtering with partial observation losses,''
  in \emph{Proc.~IEEE Conf.~Decis.~Contr.}, Paradise Island, Bahamas, 2004, pp.
  4180--4186.

\bibitem{censi11}
A.~Censi, ``Kalman filtering with intermittent observations: Convergence for
  semi-{M}arkov chains and an intrinsic performance measure,'' \emph{{IEEE}
  Trans. Automat. Contr.}, vol.~56, no.~2, pp. 376--381, Feb. 2011.

\bibitem{huadey07}
M.~Huang and S.~Dey, ``Stability of {K}alman filtering with {M}arkovian packet
  losses,'' \emph{Automatica}, vol.~43, no.~4, pp. 598--607, Apr. 2007.

\bibitem{rohmar10}
E.~Rohr, D.~Marelli, and M.~Fu, ``Statistical properties of the error
  covariance in a {K}alman filter with random measurement losses,'' in
  \emph{Proc.~IEEE Conf.~Decis.~Contr.}, Atlanta, GA USA, 2010.

\bibitem{epsshi08}
M.~Epstein, L.~Shi, A.~Tiwari, and R.~M. Murray, ``Probabilistic performance of
  state estimation across a lossy network,'' \emph{Automatica}, vol.~44,
  no.~12, pp. 3046--3053, Dec. 2008.

\bibitem{shieps10}
L.~Shi, M.~Epstein, and R.~M. Murray, ``{K}alman filtering over a
  packet-dropping network: A probabilistic perspective,'' \emph{{IEEE} Trans.
  Automat. Contr.}, vol.~55, no.~3, pp. 594--604, March 2010.

\bibitem{jingup06}
Z.~Jin, V.~Gupta, and R.~Murray, ``State estimation over packet dropping
  networks using multiple description coding,'' \emph{Automatica}, vol.~42,
  no.~9, pp. 1441--1452, Sept. 2006.

\bibitem{xiexie08}
L.~Xie and L.~Xie, ``Stability of a random {R}iccati equation with {M}arkovian
  binary switching,'' \emph{{IEEE} Trans. Automat. Contr.}, vol.~53, no.~7, pp.
  1759--1764, August 2008.

\bibitem{youfu11}
K.~You, M.~Fu, and L.~Xie, ``Mean square stability for {K}alman filtering with
  {M}arkovian packet losses,'' \emph{Automatica}, vol.~47, no.~12, pp.
  2647--2657, Dec. 2011.

\bibitem{stenfl96}
{\"{O}}.~Stenflo, ``Iterated function systems controlled by a semi-{M}arkov
  chain,'' \emph{Theory Stoch.\ Process.}, vol.~2, no.~18, pp. 305--313, 1996.

\bibitem{bouger93}
P.~Bougerol, ``Kalman filtering with random coefficients and contractions,''
  \emph{SIAM Journal on Control and Optimization}, vol.~31, no.~4, pp.
  942--959, July 1993.

\bibitem{gilber60}
E.~N. Gilbert, ``Capacity of a burst-noise channel,'' \emph{The Bell Syst.
  Tech. J.}, vol.~39, pp. 1253--1265, Sept. 1960.

\bibitem{smisei03}
S.~C. Smith and P.~Seiler, ``Estimation with lossy measurements: Jump
  estimators for jump systems,'' \emph{{IEEE} Trans. Automat. Contr.}, vol.~48,
  no.~12, pp. 2163--2171, Dec. 2003.

\bibitem{queost11}
D.~E. Quevedo, J.~{\O stergaard}, and D.~Ne\v{s}i\'c, ``Packetized predictive
  control of stochastic systems over bit-rate limited channels with packet
  loss,'' \emph{{IEEE} Trans. Automat. Contr.}, vol.~56, no.~12, pp.
  2854--2868, Dec. 2011.

\bibitem{minfra09}
P.~Minero, M.~Franceschetti, S.~Dey, and G.~N. Nair, ``Data rate theorem for
  stabilization over time-varying feedback channels,'' \emph{{IEEE} Trans.
  Automat. Contr.}, vol.~54, no.~2, pp. 243--255, Feb. 2009.

\bibitem{park11}
P.~Park, ``Modeling, analysis, and design of wireless sensor network
  protocols,'' Ph.D. dissertation, KTH Electrical Engineering, Stockholm,
  Sweden, 2011.

\bibitem{ghamos11}
A.~Ghaffarkhah and Y.~Mostofi, ``Communication-aware motino planning in mobile
  networks,'' \emph{{IEEE} Trans. Automat. Contr.}, vol.~56, no.~10, pp.
  2478--2485, Oct. 2011.

\bibitem{gudmun91}
M.~Gudmundson, ``Correlation model for shadow fading in mobile radio systems,''
  \emph{Electron. Lett.}, vol.~27, no.~23, pp. 2145--2146, Nov. 1991.

\bibitem{agrpat09}
P.~Agrawal and N.~Patwari, ``Correlated link shadow fading in multi-hop
  wireless networks,'' \emph{{IEEE} Trans. Wireless Commun.}, vol.~8, no.~8,
  pp. 4024--4036, Aug. 2009.

\bibitem{bremau99}
P.~Br\'emaud, \emph{Markov Chains}.\hskip 1em plus 0.5em minus 0.4em\relax New
  York, N.Y.: Springer, 1999.

\bibitem{cinlar75}
E.~\c{C}inlar, \emph{Introduction to Stochastic Processes}.\hskip 1em plus
  0.5em minus 0.4em\relax Prentice Hall, Englewood Cliffs, NJ, 1975.

\bibitem{mosmur09}
Y.~Mostofi and R.~M. Murray, ``To drop or not to drop: {D}esign principles for
  {K}alman filtering over wireless fading channels,'' \emph{{IEEE} Trans.
  Automat. Contr.}, vol.~54, no.~2, pp. 376--381, Feb. 2009.

\bibitem{panver07}
N.~A. Pantazis and D.~D. Vergados, ``A survey on power control issues in
  wireless sensor networks,'' \emph{{IEEE} Commun. Surv. Tutorials}, vol.~9,
  no.~4, pp. 86--107, 2007.

\bibitem{ellio63}
E.~O. Elliot, ``Estimates of error rates for codes on burst-noise channels,''
  \emph{The Bell Syst. Tech. J.}, vol.~42, pp. 1977--1997, Sept. 1963.

\bibitem{fleran04}
A.~K. Fletcher, S.~Rangan, and V.~K. Goyal, ``Estimation from lossy sensor
  data: Jump linear modeling and {K}alman filtering,'' in \emph{Proc.
  {ACM/IEEE} Int. Conf. Information Processing in Sensor Networks}, 2004, pp.
  251--258.

\bibitem{rabine89}
L.~R. Rabiner, ``A tutorial on hidden {M}arkov models and selected applications
  in speech recognition,,'' \emph{Proc. {IEEE}}, vol.~77, no.~2, pp. 257--286,
  Feb. 1989.

\bibitem{barlim08}
V.~S. Barbu and N.~Limnios, \emph{Semi-{M}arkov Chains and Hidden Semi-Markov
  Models toward Applications}.\hskip 1em plus 0.5em minus 0.4em\relax Springer,
  2008.

\bibitem{andmoo79}
B.~D.~O. Anderson and J.~Moore, \emph{Optimal Filtering}.\hskip 1em plus 0.5em
  minus 0.4em\relax Englewood Cliffs, NJ: Prentice Hall, 1979.

\bibitem{schsin07}
L.~Schenato, B.~Sinopoli, M.~Franceschetti, K.~Poolla, and S.~S. Sastry,
  ``Foundations of control and estimation over lossy networks,'' \emph{Proc.
  {IEEE}}, vol.~95, no.~1, pp. 163--187, Jan. 2007.

\bibitem{tarras76}
T.-J. Tarn and Y.~Rasis, ``Observers for nonlinear stochastic systems,''
  \emph{{IEEE} Trans. Automat. Contr.}, vol.~21, no.~4, pp. 441--448, August
  1976.

\bibitem{quenes12a}
D.~E. Quevedo and D.~Ne\v{s}i\'c, ``Robust stability of packetized predictive
  control of nonlinear systems with disturbances and {M}arkovian packet
  dropouts,'' \emph{Automatica}, accepted for publication.

\bibitem{cinlar69}
E.~\c{C}inlar, ``Markov renewal theory,'' \emph{Adv. Applied Prob.}, vol.~1,
  no.~2, pp. 123--187, 1969.

\bibitem{howard71}
R.~A. Howard, \emph{Dynamic Probabilistic Systems, Vol. II: Semi-Markov and
  Decision Processes}.\hskip 1em plus 0.5em minus 0.4em\relax John Wiley \&
  Sons, 1971.

\bibitem{kleinr76}
L.~Kleinrock, \emph{Queueing Systems}.\hskip 1em plus 0.5em minus 0.4em\relax
  John Wiley \& Sons, 1976.

\bibitem{bernst09}
D.~S. Bernstein, \emph{Matrix Mathematics}, 2nd~ed.\hskip 1em plus 0.5em minus
  0.4em\relax Princeton, N.J.: Princeton University Press, 2009.

\bibitem{kushne71}
H.~Kushner, \emph{Introduction to Stochastic Control}.\hskip 1em plus 0.5em
  minus 0.4em\relax New York, N.Y.: Holt, Rinehart and Winston, Inc., 1971.

\bibitem{meyn89}
S.~P. Meyn, ``Ergodic theorems for discrete time stochastic systems using a
  stochastic {L}yapunov function,'' \emph{SIAM Journal on Control and
  Optimization}, vol.~27, no.~6, pp. 1409--1439, Nov. 1989.

\end{thebibliography}

\appendix
\subsection{Proof of Theorem \ref{thm:multiple_sensors}}
\label{sec:netw-with-mult-1}
We  first prepare two preliminary lemmas:
\begin{lem}
\label{lem:markov}
  Suppose that Assumptions~\ref{ass:Markov} and~\ref{ass:policies} hold. Then the composite process
  $\{Z\}_{{\Nset}_0}$ 
  defined via
  \begin{equation}
\label{eq:30b}  
  Z(k)\eq \big(P(k|k-1),\Xi(k-1)\big),
\end{equation}
is a Markov chain.
\end{lem}
\begin{proof}
  Recall that
  the network state $\{\Xi\}_{\Nset_0}$ is Markovian
  and that, for given network states, the dropout processes are independent. Therefore,
   the distribution of the matrix $C(k)$  satisfies, for all $k\in\Nset_0$,
   \begin{equation*}
     \Prob\{C(k)\,|\,
     \Xi(k-1),\Xi(k-2),\dots\}
     =\Prob\{C(k)\,|\,\Xi(k-1)\}.
   \end{equation*}
The result  follows from \eqref{eq:11}, since
$\{A\}_{\Nset_0}$, $\{Q\}_{\Nset_0}$ and $\{R\}_{\Nset_0}$ are deterministic sequences. 
\end{proof}

\begin{lem}
\label{lem:bound}
  Suppose that Assumption~\ref{ass:Markov} holds and define
  \begin{equation}
    \label{eq:45}
    V_k\eq \tr P (k|k-1),
  \end{equation}
see~(\ref{eq:21}).
  Then $V_k\in\Rset_{\geq 0}$ for all $k\in\Nset_0$. Furthermore, there exists
  $W\in\Rset_{\geq 0}$,
such that,  for all $k\in\Nset_0$,  
  \begin{equation}
    \label{eq:46}
    \E\big\{ V_{k+1} \,|\, Z(k)=(P,i) \big\} \leq W + \nu_i \big(
  \|A(k)\|^2 \tr P +\tr Q(k) \big),
  \end{equation}
where $Z(k)$ is as defined
  in~\eqref{eq:30b}, and $\nu_i$ in~(\ref{eq:50}).
\end{lem}

\begin{proof}
  The fact that $V_k$ is non-negative follows directly from~(\ref{eq:45}).
  To prove~\eqref{eq:46}, it is  convenient to recall~\eqref{eq:12} and
condition as follows:
\begin{equation}
  \label{eq:31b}
\begin{split}
  \E&\big\{ V_{k+1} \,|\, Z(k) \big\}\\ 
&=\E\big\{ V_{k+1} \,|\, Z(k),r(k) = 1 \big\} \Prob\{r(k) =1\,|\,Z(k)\}\\
&\,+ \E\big\{V_{k+1} \,|\, Z(k),r(k) = 0 \big\} \Prob\{r(k) =0\,|\,Z(k)\}.
\end{split}
\end{equation}
We  next examine the cases $r(k) \in\{0,1\}$  separately.

\par 1) For $r(k)=1$,  $C(k)$ is full rank. Therefore,
 a simple  predictor for 
$x(k+1)$ given  $y(k)\subset\Ical(k)$,
is given by 
  $$\check{x}(k+1) = A(k)\big(C(k)^TC(k)\big)^{-1}C(k)^T y(k),$$
in which case
$$  \check{x}(k+1)  - x(k+1) = A(k)\big(C(k)^TC(k)\big)^{-1}C(k)^Tv(k) -w(k),$$
 where
 $
   v(k)\eq
   \begin{bmatrix}
     v_1(k)^T&v_2(k)^T&\dots &v_M(k)^T
   \end{bmatrix}^T$.
Since, by assumption, $\{A\}_{\Nset_0}$, $\{Q\}_{\Nset_0}$ and $\{R\}_{\Nset_0}$
are bounded, 
 there exists a constant $W\in\Rset_{\geq 0}$,  such that
$$\E \big\{\big(\check{x}(k+1)-x(k+1)\big) \big(\check{x}(k+1)-x(k+1)\big)^T\big\}
\preceq (W/n) I_n.$$ 
Since the Kalman filter  gives the minimum conditional error covariance matrix,
and by the fact that for any square matrix $F$, $\E \{\tr F\} = \tr \E \{F\}$,
we 
 obtain the bound\footnote{Clearly,~\eqref{eq:38b} is not a tight bound, but it
   suffices for our  purpose.}
\begin{equation}
  \label{eq:38b}
  \begin{split}
    \E&\big\{ V_{k+1} \,|\, Z(k),r(k) = 1 \big\}\Prob\{r(k)=1 | Z(k)\}\\
    &\quad\leq W
    \Prob\{r(k)=1 | Z(k)\}\leq W.
  \end{split}
\end{equation}
\par 2) For $r(k) =0$, the estimator covariance matrix $P(k+1|k)$ is
upper-bounded  by  the worst case, where $\gamma_m(k) =
0,\forall m \in\{1,2,\dots,M\}$. We, thus have 
\begin{equation}
  \label{eq:40b}
\begin{split}
  &\E\big\{V_{k+1} \,|\, Z(k)=(P,i),r(k) = 0 \big\} \leq \E\big\{V_{k+1} \,|\,
  Z(k)=(P,i),\\
  &\gamma_1(k) = \gamma_2(k)=\dots \gamma_M(k)=0\big\}= 
\tr\big(A(k)PA(k)^T+Q(k)\big)\\
  &=\tr\big(A(k)^TA(k)P\big)+\tr Q(k) \leq \|A(k)\|^2\tr P +\tr Q(k), 
\end{split}
\end{equation}
where we have used~\eqref{eq:11} and  \cite[Fact 8.12.29]{bernst09}.

\par  To calculate $\Prob \{r(k)=0\,|\,Z(k)=(P,i)\}$, we condition upon
$\Xi(k)$ and use Assumption~\ref{ass:Markov}:
\begin{equation}
\begin{split}
\label{eq:35c}
   &\Prob \{r(k)=0\,|\,Z(k)=(P,i)\}\\
   &= \! \sum_{j\in\Bset} \Prob\{r(k)=0 | P(k|k-1)=P,\Xi(k-1)=i   ,\Xi (k)=j\}\\
   &\qquad\qquad \times\,\Prob\{\Xi(k)=j\,|\,P(k|k-1)=P,\Xi(k-1)=i\}   \\
 & =  \sum_{j\in\Bset} \Prob\{r(k)= 0 \,|\, \Xi(k)=j\}\\
 &\qquad\qquad \times\,
\Prob\{\Xi (k)=j\,|\,\Xi(k-1)=i\}=\nu_i.
\end{split}
\end{equation}
The result  follows upon replacing~\eqref{eq:38b}--\eqref{eq:35c}
into~\eqref{eq:31b}.
\end{proof}

\emph{Proof of Theorem~\ref{thm:multiple_sensors}:}
  We will use a stochastic Lyapunov function approach with candidate function $V_k$
  introduced in~\eqref{eq:45}; see, e.g.,\cite{kushne71,meyn89}.  By Lemma~\ref{lem:bound}, we have
  \begin{equation}
\begin{split}
\label{eq:36b}
0\leq \E&\big\{ V_{k+1} \,|\, Z(k)=(P,i) \big\}\\
&\leq W +  \big(
  \|A(k)\|^2 \tr P +\tr Q(k) \big)\nu_i \\
&\leq\nu_i \|A(k)\|^2 V_k +\bar{\beta} \leq \rho V_k
+\bar{\beta},\quad \forall k \in \Nset_0,
\end{split}
\end{equation}
where $\rho\in[0,1)$ is as in~\eqref{eq:18b} and $$\bar{\beta} \eq W +
\max_{i\in\Nset_0} \nu_i \times \max_{k\in\Kset}\tr Q(k) \in \Rset_{\geq 0}.$$
Since~\eqref{eq:36b} holds for all $Z(k)=\big(P(k|k-1),\Xi(k-1)\big)$ and, by Lemma~\ref{lem:markov},
$\{Z\}_{\Nset_0}$  is Markovian, we can use \cite[Prop.\ 3.2]{meyn89}
 to conclude
  that~\eqref{eq:36b} is a sufficient condition for
  \begin{equation}
    \label{eq:39b}
    \begin{split}
      0\leq \E &\big\{ V_k \,|\, Z(0)=Z \big\} \leq \rho^k V_0 + \bar{\beta}
      \sum_{i=0}^{k-1}\rho^i\\
      &= \rho^k V_0 +\bar{\beta}
      \frac{1-\rho^k}{1-\rho},\quad \forall k\in\Nset_0.
    \end{split}
  \end{equation}
      Therefore, upon noting that $P(0|-1)=P_0$ is
  given,~\eqref{eq:39b} 
  gives that~\eqref{eq:14} holds with $\alpha= \rho V_0$ and
  $\beta=\bar{\beta}/(1-\rho)$. \hfill $\blacksquare$

\subsection{Proof of Theorem \ref{thm:semimarkov}}
\label{sec:proof-theor-refthm:s}
To derive our result,  we will first focus on the time-instances $\Kset$ defined
in~(\ref{eq:8}). The key property we will use is that whilst in this case $\{\Xi\}_{\Nset_0}$ is not
  Markovian, the embedded chain $\{\Xi\}_{\Kset_0}$ is Markovian. We begin by extending
  Lemmas~\ref{lem:markov} and~\ref{lem:bound} to the model in
  Assumption~\ref{ass:dwell}.
\begin{lem}
  \label{lem:markov_semi}
Suppose that Assumptions~\ref{ass:policies} and~\ref{ass:dwell} hold. Then 
  $\{Z\}_{{\Kset}_0}$ defined as in~(\ref{eq:30b}) is Markovian. 
\end{lem}
\begin{proof}
  It is easy to see that, for all $k_\ell,k_{\ell-1}\in \Kset$,
   \begin{equation*}
     \Prob\{C(k_\ell)\,|\,
     \Xi(k_{\ell}-1),\Xi(k_{\ell-1}-1),\dots\}
     =\Prob\{C(k_\ell)\,|\,\Xi(k_\ell-1)\}.
   \end{equation*}
The result now follows from  \eqref{eq:11} and upon noting that
$\{\Xi\}_{\Kset_0}$ is Markovian, and 
$\{A\}_{\Nset_0}$,  $\{Q\}_{\Nset_0}$ and $\{R\}_{\Nset_0}$ are deterministic sequences. 
\end{proof}

\begin{lem}
\label{lem:bound_semi}
Suppose that Assumption~\ref{ass:dwell} holds and consider 
$$V_\ell\eq \tr
P(k_\ell|k_\ell-1)\in\Rset_{\geq 0},\quad \ell\in \Nset_0.$$ Then there exists
$W\in\Rset_{\geq 0}$ such that,  for all
    $\ell\in\Nset_0$,
\begin{equation}
  \label{eq:34}
  \begin{split}
    \E&\big\{ V_{\ell+1} \,|\, Z(k_\ell)=(P,i) \big\}\\
    &\leq W +
    \sum_{\delta=1}^\sigma\sum_{j\in\Bset} q_{ij}\psi_{j}(\delta)
    \mu_i(k_\ell,\delta) \|\Phi(k_\ell+\delta,k_\ell)\|^2 \tr P .
  \end{split}
 \end{equation}
\end{lem}

\begin{proof}
We first 
note that, given  Assumption~\ref{ass:dwell}, the holding times $\Delta_\ell$
have conditional distribution 
\begin{equation*}
  \begin{split}
    &\Prob\{\Delta_\ell=\delta\,|\,Z(k_\ell)=(P,i)\}
    =\Prob\{\Delta_\ell = \delta \,|\,\Xi(k_\ell-1)=i \} \\
    &=\sum_{j\in\Bset}
    \Prob\{\Delta_\ell = \delta\,|\,\Xi(k_\ell)=j,\Xi(k_{\ell}-1)=i \} \\
    &\qquad\times\Prob\{\Xi(k_{\ell})=j\,|\,\Xi(k_{\ell}-1)=i\}\\
    &=\sum_{j\in\Bset}
    \Prob\{\Delta_\ell = \delta\,|\,\Xi(k_\ell)=j\} 
    \Prob\{\Xi(k_{\ell})=j\,|\,\Xi(k_{\ell-1})=i\}\\
    &=\sum_{j\in\Bset}
    q_{ij}\psi_{j}(\delta),
  \end{split}
\end{equation*}
thus, 
\begin{equation}
\label{eq:48}
\begin{split}
  \E&\big\{ V_{\ell+1} \,|\, Z(k_\ell)=(P,i) \big\}\\ 
  &=
  \sum_{\delta=1}^\sigma\sum_{j\in\Bset} q_{ij}\psi_{j}(\delta) \E\big\{
  V_{\ell+1} \,|\, Z(k_\ell)=(P,i),\Delta_\ell=\delta \big\} .
\end{split}
\end{equation}
We next condition on  $\varrho(k_\ell)$ defined in~(\ref{eq:12b}) to obtain
\begin{equation}
  \label{eq:37}
  \begin{split}
  \E&\big\{ V_{\ell+1} \,|\, Z(k_\ell),\Delta_\ell \big\}
\leq \E\big\{ V_{\ell+1} \,|\, Z(k_\ell),\Delta_\ell,\varrho(k_\ell) = 1 \big\}\\
&+ \E\big\{V_{\ell+1} \,|\, Z(k_\ell),\Delta_\ell,\varrho(k_\ell) = 0 \big\}\\
&\qquad\times\Prob\{\varrho(k_\ell) =0\,|\,Z(k_\ell),\Delta_\ell\} 
\end{split}
\end{equation}
and use~(\ref{eq:33}) to write
\begin{equation}
  \label{eq:38}
  \begin{split}
    \Prob&\{\varrho(k_\ell) =0\,|\,Z(k_\ell)=(P,i),\Delta_\ell=\delta\}\\ 
    &=
    \Prob\{\varrho(k_\ell) =0\,|\,\Xi(k_\ell-1)=i, \Delta_\ell=\delta\}\\
    &=\sum_{j\in\Bset} \Prob\{\varrho(k_\ell)
    =0\,|\,\Xi(k_\ell)=j,\Xi(k_\ell-1)=i, \Delta_\ell=\delta\}\\
    &\qquad\times\Prob\{\Xi(k_\ell)=j|\Xi(k_\ell-1)=i\}\\
    &=\sum_{j\in\Bset}q_{ij} \Prob\{\varrho(k_\ell)
    =0\,|\,\Xi(k_\ell)=j, \Delta_\ell=\delta\}
    =\mu_i(k_\ell,\delta).
  \end{split}
\end{equation}
 In what follows, we examine the cases $\varrho(k_\ell) \in\{0,1\}$  separately.
\par 1) For $\varrho(k_\ell)=1$,  
 a suboptimal  predictor for 
$x(k_{\ell+1})$ which only uses the received measurements $\{y\}_{k_\ell}^{k_{\ell+1}-1} \subset\Ical(k_{\ell+1}-1)$
is given by 
\begin{equation*}
  \begin{split}
    \check{x}(k_{\ell+1}) &=
    \Phi(k_{\ell+1},k_\ell)\big(\mathcal{O}(k_{\ell+1}-1,k_\ell)^T
    \mathcal{O}(k_{\ell+1}-1,k_\ell)\big)^{-1}\\
    &\qquad\times\mathcal{O}(k_{\ell+1}-1,k_\ell)^T
    \begin{bmatrix}
      y(k_\ell)\\y(k_{\ell+1})\\
      \vdots\\y(k_{\ell+1}-1)
    \end{bmatrix}.
  \end{split}
\end{equation*}
Since $\{A\}_{\Nset_0}$, $\{Q\}_{\Nset_0}$ and $\{R\}_{\Nset_0}$
are assumed bounded, the matrices $\Phi(k_{\ell+1},k_\ell)$ and
$\mathcal{O}(k_{\ell+1}-1,k_\ell)$ are bounded also. Consequently, 
the  estimation error
covariance of $\check{x}(k_{\ell+1})$ is bounded. Thus, (due to optimality) in
case of the 
Kalman filter~(\ref{eq:11}), there exists 
$W_1\in\Rset_{\geq 0}$ such that
\begin{equation}
  \label{eq:39}
  \E\big\{ V_{\ell+1} \,|\, Z(k_\ell),\Delta_\ell,\varrho(k_\ell) = 1 \big\}\leq W_1.
\end{equation}
\par 2) For  $\varrho(k_\ell) =0$, the estimation error
covariance matrix $P(k_{\ell+1}|k_{\ell+1}-1)$ is 
upper-bounded  by that resulting from the worst case, i.e., where
 $\gamma_m(k) =
0$, $\forall m \in\{1,2,\dots,M\}$, $\forall
k\in\{k_\ell,k_\ell+1,\dots,k_{\ell+1}-1\}$.
Use of~(\ref{eq:11}) then gives
\begin{equation*}
  \begin{split}
    \E&\big\{V_{\ell+1} \,|\, Z(k_\ell)=(P,i),\Delta_\ell=\delta,\varrho(k_\ell)
    = 0 \big\}\\
    &\leq
    \E\big\{V_{\ell+1} \,|\, 
  Z(k_\ell)=(P,i),\Delta_\ell=\delta,C(k) = 0,\\
  &\qquad\qquad \forall
  k\in\{k_\ell,\dots,k_\ell +\delta-1\}\big\}\\ 
&=\tr\big(P(k_\ell+\delta |k_\ell+\delta-1) \,|\, P(k_\ell|k_\ell-1)=P, C(k) =
0,\\  &\qquad\qquad\forall 
k\in\{k_\ell,\dots,k_\ell +\delta-1\} \big)\\
&=\tr\bigg(\Phi(k_\ell+\delta,k_\ell) P \Phi(k_\ell+\delta,k_\ell)^T\\
& \;+
\sum_{j=1}^\delta \Phi(k_\ell+\delta,k_\ell+j)
Q(k_\ell+j-1)\Phi(k_\ell+\delta,k_\ell+j)^T\!\bigg) .
  \end{split}
\end{equation*}
Since $\{A\}_{\Nset_0}$ and $\{Q\}_{\Nset_0}$ are bounded, there exists
$W_2\in\Rset_{\geq 0}$ such that
\begin{equation}
  \label{eq:53}
  \begin{split}
    \E\big\{V_{\ell+1} \,|\, Z&(k_\ell)=(P,i),\Delta_\ell=\delta,\varrho(k_\ell)
    = 0 \big\}\\
    &\leq\tr\big(\Phi(k_\ell+\delta,k_\ell) P
    \Phi(k_\ell+\delta,k_\ell)^T\big) + W_2 \\
    &=\tr\big(  \Phi(k_\ell+\delta,k_\ell)^T\Phi(k_\ell+\delta,k_\ell)P\big)+W_2\\
    &\leq \|\Phi(k_\ell+\delta,k_\ell)\|^2 \tr P + W_2,
  \end{split}
\end{equation}
where we have used \cite[Fact 8.12.29]{bernst09}.
\par Substitution of~(\ref{eq:38})--(\ref{eq:53}) into~(\ref{eq:37}) provides
\begin{equation*}
  \begin{split}
    \E\big\{ V_{\ell+1} &\,|\, Z(k_\ell)=(P,i),\Delta_\ell=\delta \big\}\\
    &\leq
    W_1+\mu_i(k_\ell,\delta) \big( \|\Phi(k_\ell+\delta,k_\ell)\|^2 \tr P + W_2
    \big),
  \end{split}
\end{equation*}
thus,~(\ref{eq:48}) yields the bound
\begin{equation*}
  \begin{split}
    \E&\big\{ V_{\ell+1} \,|\, Z(k_\ell)=(P,i) \big\}    \leq
    \sum_{\delta=1}^\sigma\sum_{j\in\Bset} q_{ij}\psi_{j}(\delta)\\
    &\quad\big(
    W_1+\mu_i(k_\ell,\delta) \big( \|\Phi(k_\ell+\delta,k_\ell)\|^2 \tr P + W_2
    \big)\big)
  \end{split}
\end{equation*}
from where the result follows.
\end{proof}
To prove Theorem~\ref{thm:semimarkov}, we recall that, by
Lemma~\ref{lem:bound_semi} and~(\ref{eq:18c}), 
\begin{equation}
  \label{eq:17}
  0\leq  \E \{ \tr P(k_{\ell+1}|k_{\ell+1}-1) \,|\, Z(k_\ell)  \} \leq W 
+ \rho^\sigma
  \tr P(k_{\ell}|k_{\ell}-1), 
\end{equation}
 for all $\ell \in \Nset_0$. Thus, following as in the proof of Theorem~\ref{thm:multiple_sensors}, it is
easy to
establish exponential boundedness \emph{at the time instants $k_\ell\in\Kset$},
i.e., there exist $\alpha_1, \beta_1\in\Rset_{\geq 0}$ such that:
  \begin{equation}
    \label{eq:35}
    \E \{ \tr P(k_\ell|k_\ell-1)\} \leq \alpha_1 \rho^{\ell\sigma}
    +\beta_1\leq \alpha_1 \rho^{k_\ell}
    +\beta_1, \; \forall k_\ell \in\Kset,
  \end{equation}
where we have used the fact that $\Delta_\ell\leq \sigma$, thus, $k_\ell\leq
\ell\sigma$. 
To establish exponential boundedness \emph{at all instants $k\in\Nset_0$}, we
note that, similar to~(\ref{eq:53}), there exist finite $\alpha_2\geq 1$,
$W_3\in\Rset_{\geq 0}$ 
such that, for all $ t\in\{0,1,\dots,\Delta_\ell-1\}$, 
\begin{equation*}
  \begin{split}
    \tr P(k_\ell +t |k_\ell+t-1) &\leq \|\Phi(k_\ell +t,k_\ell)\|^2 \tr
    P(k_\ell|k_\ell-1) +W_3 \\
    &\leq \alpha_2 \rho^{t} \tr P(k_\ell|k_\ell-1) +W_3.
  \end{split}
\end{equation*}
 Taking expectation and using~(\ref{eq:35})
gives
\begin{equation*}
  \begin{split}
    \E&\{\tr P(k_\ell +t |k_\ell+t-1) \} \leq \alpha_1\alpha_2 \rho^{k_\ell+t}
    + \alpha_2\beta_1\rho^t +W_3\\
    &\leq \alpha_1\alpha_2 \rho^{k_\ell+t} +
    \alpha_2\beta_1 +W_3,\quad \forall t\in\{0,1,\dots,\Delta_\ell-1\}.
  \end{split}
\end{equation*}
The latter expression and~(\ref{eq:35}) show exponential boundedness:
\begin{equation*}
  \E\{\tr P(k|k-1) \} \leq  \alpha_1\alpha_2  {\rho}^k +
  \alpha_2\beta_1 +W_3,\quad
  \forall k \in\Nset_0. 
\end{equation*}
This proves Theorem~\ref{thm:semimarkov}.

\begin{biography}[{\includegraphics[width=1in,height=1.25in,clip,keepaspectratio]{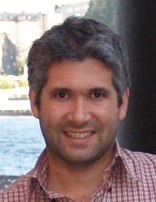}}]{\bf    Daniel E.\ Quevedo} (S'97--M'05) received Ingeniero Civil Electr\'onico
and Magister en Ingenier\'{\i}a Electr\'onica degrees from the Universidad
T\'ecnica Federico Santa Mar\'{\i}a, Valpara\'{\i}so, Chile in 2000.  In
2005, he received the Ph.D.~degree from The University of Newcastle,
Australia, where he is currently an Associate Professor.  He has been a
visiting researcher at various institutions, including Uppsala University,
Sweden, KTH Stockholm, Sweden, Aalborg University, Denmark, Kyoto University,
Japan, and INRIA Grenoble, France.
\par Dr.\ Quevedo was supported by a full scholarship from the alumni
association during his time at the Universidad 
T\'ecnica Federico Santa Mar\'{\i}a and received several university-wide
prizes upon graduating. He received the IEEE Conference on Decision and
Control Best Student Paper Award in 2003 and was also a finalist  in 
2002.  In
2009, he was awarded a five-year Australian Research Fellowship.   His  research
interests include several areas of automatic control, signal processing, and power
electronics. 
\end{biography}

\begin{biography}[{\includegraphics[width=1in,height=1.25in,clip,keepaspectratio]{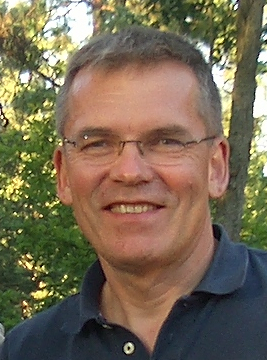}}]{\bf
    Anders Ahl\'en} (S'80--M'84--SM'90)  is full professor and holds
the chair in Signal Processing at Uppsala University where he is
also the head of the Signals and Systems Division of The Department of
Engineering Sciences. He was born in
Kalmar, Sweden, and received the PhD degree in Automatic Control
from Uppsala University. He was with the Systems and
Control Group, Uppsala University from 1984-1992 as an Assistant
and Associate Professor  in
Automatic Control. During 1991 he was a visiting researcher at the
Department of Electrical and Computer Engineering, The University
of Newcastle, Australia. He was a visiting professor at the same
university in 2008. In 1992 he was appointed Associate
Professor of Signal Processing at Uppsala University.
During 2001-2004 he was
the CEO of Dirac Research AB, a company offering state-of-the-art
audio signal processing solutions. He is currently the chairman of
the board of the same company. Since 2007 he has been a member of the Uppsala VINN Excellence Center for
Wireless Sensor Networks,
WISENET. His research interests, which
include Signal Processing, Communications and Control, are
currently focused on Signal Processing for Wireless
Communications, Wireless Systems Beyond 3G, Wireless Sensor
Networks, Wireless Control, and Audio Signal Processing.
\par From 1998 to 2004 he was the Editor of Signal and Modulation Design
for the IEEE Transactions on Communications.
\end{biography}

\begin{biography}[{\includegraphics[width=1in,height=1.25in,clip,keepaspectratio]{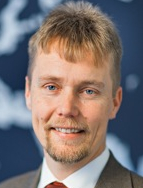}}]{\bf 
Karl H.\ Johansson} (SM'08) is Director of the KTH ACCESS Linnaeus Centre and
Professor at the School of Electrical Engineering, Royal Institute of
Technology, Sweden. He is a Wallenberg Scholar and has held a Senior Researcher
Position with the Swedish Research Council. He received MSc and PhD degrees in
Electrical Engineering from Lund University. He has held visiting positions at
UC Berkeley (1998-2000) and California Institute of Technology (2006-2007). His
research interests are in networked control systems, hybrid and embedded
control, and control applications in automotive, automation and communication
systems. He was a member of the IEEE Control Systems Society Board of Governors
2009 and the Chair of the IFAC Technical Committee on Networked Systems
2008-2011. He has been on the Editorial Boards of Automatica (2003-2006) and
IEEE Transactions on Automatic Control (2008-2010), and is currently on the
Editorial Boards of IET Control Theory and Applications and the International
Journal of Robust and Nonlinear Control. He was the General Chair of the
ACM/IEEE Cyber-Physical Systems Week (CPSWeek) 2010 in Stockholm. He has served
on the Executive Committees of several European research projects in the area of
networked embedded systems. In 2009, he received the Best Paper Award of the
IEEE International Conference on Mobile Ad-hoc and Sensor Systems. He was
awarded an Individual Grant for the Advancement of Research Leaders from the
Swedish Foundation for Strategic Research in 2005. He received the triennial
Young Author Prize from IFAC in 1996 and the Peccei Award from the International
Institute of System Analysis, Austria, in 1993. He received Young Researcher
Awards from Scania in 1996 and from Ericsson in 1998 and 1999. 
\end{biography}

\end{document}